\documentclass[11pt,a4paper]{article}
\usepackage[utf8]{inputenc}
\usepackage[T1]{fontenc}
\usepackage{amsmath,amssymb,amsthm}
\usepackage{mathpazo}
\usepackage[margin=2.5cm]{geometry}
\usepackage{mathrsfs}
\usepackage{booktabs}
\usepackage{microtype}
\usepackage{enumitem}
\usepackage{hyperref}
\hypersetup{colorlinks=true,linkcolor=blue,urlcolor=blue,citecolor=blue}
\usepackage{setspace}
\newtheorem{theorem}{Theorem}[section]
\newtheorem{proposition}[theorem]{Proposition}
\newtheorem{lemma}[theorem]{Lemma}
\newtheorem{corollary}[theorem]{Corollary}

\newtheorem{principle}[theorem]{Principle}
\newtheorem{definition}[theorem]{Definition}
\newtheorem{remark}[theorem]{Remark}
\newtheorem{example}[theorem]{Example}

\newcommand{\E}{\mathrm{E}}

\newcommand{\R}{\mathbb{R}}

\newcommand{\cS}{\mathcal{S}}
\newcommand{\wm}[1]{{}^{(\varphi)}\!m_{#1}}
\newcommand{\wk}[1]{{}^{(\varphi)}\!\kappa_{#1}}
\newcommand{\wg}{{}^{(\varphi)}\!g}
\newcommand{\wG}{{}^{(\varphi)}\!G}

\newcommand{\Fseven} {\fontsize{7}{11}\selectfont  }

\newcommand{\Ften} {\fontsize{10}{11}\selectfont  }

\usepackage{fancyhdr}
\title{Transversality and Geometric Regularisation\\
in Distributional Statistical Models}

\author{Rodrigo Labouriau\\[4pt]
\small Department of Mathematics, Aarhus University\\
\small \texttt{rodrigo.labouriau@math.au.dk} and
\small \texttt{rodrigo.labouriau@rlstatlab.com}}

\date{Spring 2026}

\begin{document}
\maketitle

\begin{abstract}
The distributional statistical framework represents a probability law by a
tempered distribution $T$---on the same footing as a density or a
characteristic function---and extracts information from it through a positive,
rapidly decaying kernel $\varphi$ that serves as a measurement instrument
rather than as part of the law.  This paper develops the thesis that the kernel acts
as a \emph{geometric regulariser}, placing the statistical model in a
generic (transversal) position relative to degeneracy loci that encode
non-identifiability, singular information, moment indeterminacy, and
representation failure.

Using the classical transversality theorems of Whitney, Thom, and
Mather, and their infinite-dimensional extensions due to Smale and
Abraham, we prove a finite-dimensional weak transversality theorem
showing that, for a generic kernel in any sufficiently rich family,
the kernel-induced feature map avoids degeneracy strata of
sufficiently high codimension.  The core results are developed for
parametric models with finite-dimensional parameter spaces; for the
infinite-dimensional case we give a \emph{conditional} transversality
theorem that, under a Fredholm hypothesis on the feature map, reduces the
genericity question to the same finite-rank mechanism, and we verify its
hypotheses in closed form on a non-dominated Gaussian sequence model.  We then establish concrete,
verifiable conditions---formulated as rank conditions on the
Jacobian of the joint feature map---under which the transversality
hypothesis can be checked for specific model classes, and verify
these conditions for location families, the log-normal, Stein
discrepancies, and graphical models.  We extend the degeneracy
classification to include representation degeneracy (Type~0) for
models without closed-form densities---including genuinely non-dominated
families, which possess no likelihood---and identify higher-order
instabilities (Type~IV) arising in non-chordal graphical models.
We show that these phenomena---including identifiability, robustness,
moment determinacy, Fisher information regularity, Stein discrepancy,
inferential separation, and the Behrens--Fisher problem---admit a
unified geometric interpretation as transversality conditions on the
feature map.

This paper serves as a geometric companion to a series of papers
developing the distributional statistical
framework, providing the transversality-theoretic
foundation that unifies their separate contributions.


\end{abstract}

\newpage
\setstretch{0.9}
\Fseven
\tableofcontents
\newpage
\normalsize
\setstretch{1.1}


\section{Introduction}
\label{sec:intro}
Many statistical models of practical and theoretical importance fall outside the
scope of classical moment-based methodology. Heavy-tailed distributions may lack
finite moments, while others, such as the log-normal, are moment-indeterminate.
In such settings, fundamental concepts—identifiability, Fisher information, and
likelihood-based inference—become ill-defined or unstable.

A recent line of work represents a probability law by a tempered
distribution $T$---a characterisation on the same footing as a density or a
characteristic function---observed through a positive, rapidly decaying
kernel $\varphi$~\cite{A}: the kernel is the instrument through which the law
is probed, not part of the law. Expectations and moments are defined
via the pairing $\langle T, g\varphi\rangle$, yielding well-defined weak moments
and weak characteristic functions of all orders.

Empirically, the introduction of a kernel resolves several difficulties
simultaneously: it restores identifiability in moment-indeterminate models,
regularises information quantities, and induces robustness through bounded
influence functions. The purpose of this paper is to provide a geometric
explanation of these phenomena.

Our main thesis is that the kernel---the instrument through which the law is
observed---acts as a generic perturbation in the sense of transversality
theory: the freedom to choose the instrument is precisely the genericity that
dissolves the degeneracies. Given a parametric model $\{T_\theta: \theta\in\Theta\}$ and a kernel
$\varphi$, the weak framework induces a feature map
\[
\Phi_\varphi : \Theta \to \mathcal F,
\]
where $\mathcal F$ is a feature space of weak moments or weak characteristic
functions. Statistical pathologies correspond to geometric degeneracies of this
map: non-identifiability to self-intersections, singular information to rank
deficiency, and moment indeterminacy to failure of separation.

We show that, under suitable regularity conditions, these degeneracies are
non-generic. In particular, we prove a finite-dimensional transversality result
(Theorem~\ref{thm:main}) showing that, for a generic choice
of kernel in a sufficiently rich family, the feature map is transversal to the
degeneracy strata, and hence avoids them when their codimension exceeds the
model dimension.  A key limitation of the present development is that it applies
to \emph{parametric} models, i.e.\ models with finite-dimensional parameter
spaces $\Theta \subset \R^p$; the extension to semiparametric and nonparametric
settings, where $\Theta$ is infinite-dimensional, requires different analytical
tools and is discussed in Section~\ref{sec:infinite_dim}.
We then move beyond the abstract genericity statement and
establish concrete, verifiable conditions---formulated as rank conditions on the
Jacobian of the joint map---under which the transversality hypothesis can be
checked directly from the model structure.  We verify these conditions for
location families, the log-normal, Stein discrepancies, and graphical models.

This perspective leads to a unified interpretation of several phenomena:
identifiability corresponds to injectivity of the feature map, information
regularity to immersivity, and robustness to boundedness of the induced metric.
We extend the degeneracy classification to include representation degeneracy,
arising in models without closed-form densities, and higher-order instabilities
arising in non-chordal graphical models.

The paper is organised as follows.
Section~\ref{sec:distributional} recalls the distributional framework.
Section~\ref{sec:transversality} reviews transversality theory.
Section~\ref{sec:finite_dim_thm} presents the finite-dimensional
transversality theorem.
Section~\ref{sec:transversality_condition_dim} establishes verifiable
conditions---expressed as rank conditions on the Jacobian of the joint
feature map---under which the transversality hypothesis holds, and
verifies them for location families, the log-normal, Stein
discrepancies, and graphical models.
Section~\ref{sec:infinite_dim} discusses infinite-dimensional
extensions.
Section~\ref{sec:degeneracies} formulates the degeneracy
stratification.
Section~\ref{sec:examples} provides concrete examples, including the
log-normal and the Cauchy.
Sections~\ref{sec:stein} and~\ref{sec:behrens_fisher} connect the
framework to Stein’s method and the Behrens--Fisher problem.
Section~\ref{sec:classification} classifies singularities by
codimension.
Section~\ref{sec:unification} develops the unifying perspective,
including the interpretation of inferential separation as
transversality to the nuisance-parameter fibration.

\section{The Distributional Framework}
\label{sec:distributional}

We briefly recall the essentials of the distributional statistical
framework; see~\cite{A,B} for the full development.

\subsection{Distribution--kernel pairs}

Let $\cS(\R^d)$ denote the Schwartz space of rapidly decaying smooth
functions and $\cS'(\R^d)$ its dual (tempered distributions).  A
\emph{distributional statistical model} is a parametric family
$\{T_\theta : \theta\in\Theta\}\subset\cS'(\R^d)$, where
$\Theta\subset\R^p$ is open.  A \emph{kernel} is a function
$\varphi\in\cS(\R^d)$ with $\varphi>0$.

Here $T_\theta$ represents the probability law---a characterisation on the
same footing as a density, a distribution function, or a characteristic
function---and the kernel $\varphi$ is the instrument through which it is
observed, not part of the law.
Expectations are defined through distributional pairings:
${}^{(\varphi)}\E_\theta[g] = \langle T_\theta, g\varphi\rangle$.

\begin{remark}[Singularity as differentiated regularity]
\label{rem:structure_theorem}
The use of tempered distributions should not be interpreted as
introducing arbitrarily pathological objects.  By the classical
structure theorem (see, e.g., Strichartz~\cite{Strichartz2003},
Section~6.3), every tempered distribution can be represented as a
finite sum of derivatives of continuous functions with at most
polynomial growth.  Singular probabilistic behaviour---point masses,
jumps, heavy tails---arises when ordinary functions are differentiated
in the weak sense.  The kernel~$\varphi$ acts as a regularising
observational device that converts these differentiated structures into
stable scalar quantities.  From this perspective, geometric
degeneracies (Section~\ref{sec:degeneracies}) correspond to unstable
configurations of differentiated regularity, while kernel perturbations
provide a mechanism for regularising them.
\end{remark}

\begin{remark}[Continuity and perturbation of weak features]
\label{rem:continuity_perturbation}
Tempered distributions act \emph{continuously} on Schwartz space: if
$\varphi_j\to\varphi$ in $\cS(\R^d)$, then
$\langle T,\varphi_j\rangle\to\langle T,\varphi\rangle$ for every
$T\in\cS'(\R^d)$.  Consequently the weak moments, the weak characteristic
function, and the feature map $\Phi_\varphi$ depend continuously---indeed
smoothly, for kernels varying in a smooth family---on the instrument.  This
is the analytic foundation of the transversality programme: a generic
perturbation of the instrument produces a controlled perturbation of the
feature map, so the genericity arguments of
Sections~\ref{sec:finite_dim_thm}--\ref{sec:degeneracies} act on stable,
continuously varying objects rather than on the (possibly singular) classical
density.
\end{remark}

\subsection{Weak moments and characteristic functions}

The \emph{weak moment of order~$j$} is
\[
  \wm{j}(\theta)
  \;:=\;
  \langle T_\theta,\, x^j\,\varphi(x)\rangle
  \;=\;
  \E_\theta\bigl[X^j\,\varphi(X)\bigr],
\]
where the parenthesised left superscript $(\varphi)$ records the
dependence on the kernel, following the notational convention of the
series.  The \emph{weak characteristic function} is
${}^{(\varphi)}\phi_\theta(u) = \langle T_\theta, e^{iux}\varphi(x)
\rangle$.

All weak moments and the weak characteristic function are
well-defined for all orders and all~$u$, because the kernel provides
the necessary integrability~\cite[Proposition~2.17]{A}.

\subsection{Weak cumulants}

The kernel defines a tilted probability
$p_\varphi(x;\theta) = f(x;\theta)\varphi(x)/\wm{0}(\theta)$.
The \emph{weak cumulants} $\wk{j}$ are the cumulants of $X$ under
this tilted distribution.  The weak cumulant generating function
${}^\varphi\!K(t) = \log\E_\theta[e^{tX}\varphi(X)] - \log\wm{0}$
is entire (converges for all $t\in\R$), even when the classical
cumulant generating function does not exist~\cite{D}.

\subsection{Distributional metric}

The distributional information metric~\cite{C,E} is the $p\times p$
tensor
\[
  \wG_{ab}^{(J)}(\theta)
  \;=\;
  \sum_{j=0}^{J}
  \frac{\partial \wm{j}}{\partial\theta_a}\,
  \frac{\partial \wm{j}}{\partial\theta_b}.
\]
This defines a Riemannian metric on $\Theta$.  We will see that this
tensor is the first fundamental form of the feature-map immersion
(Section~\ref{sec:info_geom}).

\section{Transversality Theory}
\label{sec:transversality}

\subsection{Transversality}

Let $M$ and $N$ be smooth manifolds and $S\subset N$ a submanifold
of codimension~$c$.

\begin{definition}[Transversality]
\label{def:transversality}
A smooth map $f:M\to N$ is \emph{transversal} to $S$, written
$f\pitchfork S$, if for every $x\in f^{-1}(S)$,
$Df_x(T_x M) + T_{f(x)} S = T_{f(x)} N$.
\end{definition}

\begin{proposition}
\label{prop:consequences}
\begin{enumerate}[label=\textnormal{(\alph*)}]
\item If $f\pitchfork S$, then $f^{-1}(S)$ is a smooth submanifold
  of $M$ of codimension~$c$ (or empty).
\item If $c > \dim M$, then $f\pitchfork S$ iff $f(M)\cap S =
  \emptyset$.
\end{enumerate}
\end{proposition}

\subsection{Thom's theorem and the parametric version}

\begin{theorem}[Thom {\cite{Thom1954}}]
\label{thm:thom}
The set $\{f\in C^\infty(M,N) : f\pitchfork S\}$ is residual (and
open dense if $S$ is closed) in the Whitney $C^\infty$ topology.
\end{theorem}

\begin{theorem}[Parametric Transversality
  {\cite{GuilleminPollack1974,Hirsch1976}}]
\label{thm:parametric}
Let $F:M\times\Lambda\to N$ be smooth with $F\pitchfork S$.  Then
$F_\lambda\pitchfork S$ for a residual set of $\lambda\in\Lambda$.
\end{theorem}

\subsection{Jet spaces}

\begin{definition}[{\cite{Ehresmann1951}}]
The $r$-jet of $f:M\to N$ at $x$ is
$j^r f(x) = (x, f(x), Df(x), \ldots, D^rf(x))
\in J^r(M,N)$.
\end{definition}

\begin{theorem}[Multijet Transversality
  {\cite{Mather1970,GolubitskyGuillemin1973}}]
\label{thm:multijet}
For any submanifold $W\subset J^r(M,N)$, the set of $f$ with
$j^rf\pitchfork W$ is residual.
\end{theorem}

\subsection{Whitney stratifications}

Degeneracy sets are typically stratified.  A \emph{Whitney
stratification}~\cite{Whitney1965} is a partition into smooth strata
satisfying Whitney's conditions~(a) and~(b).  Thom's theorem
extends to stratified targets.

\subsection{The Thom--Mather classification}

The Thom--Mather theory~\cite{Mather1970,Mather1971,
GolubitskyGuillemin1973} classifies the generic singularities of
smooth maps.  For $f:\R^m\to\R^n$, the generic singularities form a
hierarchy (folds, cusps, swallowtails, \ldots) indexed by codimension,
exhaustive for ``nice dimensions.''

\begin{remark}[Continuity and geometric stability]
\label{rem:continuity_perturbation}
The analytical basis of the kernel perturbation mechanism is the
continuity of tempered distributions on Schwartz space: if
$\varphi_n \to \varphi$ in $\cS(\R^d)$ and $T \in \cS'(\R^d)$, then
$\langle T, \varphi_n \rangle \to \langle T, \varphi \rangle$.  Hence
perturbing the kernel induces a controlled perturbation of the weak
features used to define the feature map.  Transversality may then be
understood as a condition ensuring that such perturbations remove
tangencies or degeneracies in a stable way, rather than merely
relabelling the coordinates in which the degeneracy is expressed.
This connects the analytical regularisation supplied by the kernel with
the geometric regularisation supplied by transversality: the kernel
smooths the distributional object at a chosen observational scale,
while transversality ensures that the resulting feature map meets the
relevant degeneracy strata non-tangentially.
\end{remark}

\section{A Finite-Dimensional Weak Transversality Theorem}
\label{sec:finite_dim_thm}

We now state and prove the main result, which captures the core
geometric mechanism in a clean finite-dimensional setting.

\begin{theorem}[Finite-dimensional weak transversality]
\label{thm:main}
Let $\Theta \subset \R^p$ and $\Lambda \subset \R^q$ be open sets,
and let $\{T_\theta:\theta\in\Theta\}\subset \cS'(\R)$ be a $C^r$
parametric distributional model.  Let
$\lambda \mapsto \varphi_\lambda \in \cS(\R)$ be a $C^r$
finite-dimensional family of positive Schwartz kernels.

Fix moment orders $0\leq j_0<\cdots<j_K$, and define the finite
weak moment feature map
\[
  \Phi_\lambda:\Theta\to \R^{K+1},
  \qquad
  \Phi_\lambda(\theta)
  =
  \bigl(
  {}^{(\varphi_\lambda)}\!m_{j_k}(\theta)
  \bigr)_{k=0}^{K}.
\]
Assume that the joint map
$F(\theta,\lambda)=\Phi_\lambda(\theta)$ is $C^r$ and transversal to
a smooth submanifold $D\subset \R^{K+1}$.

Then, for a residual subset $\Lambda_D\subset\Lambda$, the
restricted feature map $\Phi_\lambda$ is transversal to $D$ for
every $\lambda\in\Lambda_D$.

Consequently, for generic $\lambda$, the degeneracy set
$\Phi_\lambda^{-1}(D)$ is either empty or a smooth submanifold of
$\Theta$ of codimension $\mathrm{codim}(D)$.  If
$\mathrm{codim}(D) > p$, then
$\Phi_\lambda(\Theta)\cap D = \varnothing$ for generic~$\lambda$.
\end{theorem}

\begin{proof}
Apply the parametric transversality theorem
(Theorem~\ref{thm:parametric}) to $F:\Theta\times\Lambda\to\R^{K+1}$
and the submanifold $D$.  Since $F\pitchfork D$ by assumption, the
conclusion follows.  The codimension statement uses
Proposition~\ref{prop:consequences}(b).
\end{proof}

\begin{remark}
The assumption $F \pitchfork D$ can be verified in practice using the
differential criteria developed in
Section~\ref{sec:transversality_condition_dim}, where it is shown to
reduce to rank conditions on the parameter and kernel derivatives of
the joint feature map.  Concrete verifications for location families,
the log-normal, Stein discrepancies, and graphical models are given
there.
\end{remark}

\begin{corollary}[Generic full-rank weak information]
\label{cor:full_rank}
Under the hypotheses of Theorem~\ref{thm:main}, if $K+1\geq p$ and
the joint first-jet map $j^1F$ is transversal to the rank-degeneracy
stratum $\Sigma^1 = \{j^1f:\mathrm{rank}(Df)<p\}$, then for generic
$\lambda$ the weak information matrix
${}^{(\varphi_\lambda)}\!G(\theta) = D\Phi_\lambda(\theta)^\top
D\Phi_\lambda(\theta)$ is nonsingular except possibly on a
submanifold of positive codimension.  If $K+1 > 2p-1$, then
generically $\det {}^{(\varphi_\lambda)}\!G(\theta)>0$ for all
$\theta\in\Theta$.
\end{corollary}

\begin{remark}
Even the one-parameter Gaussian scale family
$\varphi_s(x) = (2\pi s^2)^{-1/2}e^{-x^2/(2s^2)}$, $s>0$
($q=1$), suffices to destroy many degeneracies, as the log-normal
example demonstrates.
\end{remark}

Theorem 4.1 provides a genericity result under the transversality condition
$F \pitchfork D$. In applications, it is therefore essential to identify
conditions under which this assumption can be verified. This is the purpose
of the next section.

\section{Verifiable Transversality Conditions}
\label{sec:transversality_condition_dim}

The finite-dimensional transversality theorem
(Theorem~\ref{thm:main}) provides a powerful genericity result, but it
assumes that the joint map
$F(\theta,\lambda) = \Phi_\lambda(\theta)$ is transversal to a given
degeneracy stratum~$D$.  In applications, it is desirable to replace
this assumption by conditions that can be verified directly from the
structure of the model.

In this section we formulate such conditions, expressing
transversality in terms of derivatives of the weak feature map, and
verify them for several important classes of models.

\subsection{A component-wise transversality criterion}

Let $F : \Theta \times \Lambda \to \R^{K+1}$ be the joint weak
feature map,
\[
  F(\theta,\lambda)
  \;=\;
  \bigl(
  {}^{(\varphi_\lambda)}m_{j_0}(\theta),\;
  \ldots,\;
  {}^{(\varphi_\lambda)}m_{j_K}(\theta)
  \bigr).
\]
Let $D \subset \R^{K+1}$ be a smooth submanifold of codimension~$c$
and let $y = F(\theta,\lambda) \in D$.  Write $N_y D$ for the normal
space to $D$ at~$y$, i.e.\ $N_y D = (T_y D)^\perp$ in~$\R^{K+1}$.

Because $\Theta\times\Lambda$ is a product, the derivative
$DF(\theta,\lambda)$ admits a natural decomposition into a
\emph{model component} and a \emph{kernel component}:
\begin{equation}\label{eq:DF_decomposition}
  DF(\theta,\lambda)\;=\;
  \bigl(\,D_\theta F,\;\; D_\lambda F\,\bigr),
\end{equation}
where $D_\theta F = D_\theta\Phi_\lambda(\theta)$ and $D_\lambda F$
collects the derivatives with respect to the kernel parameters.

\begin{lemma}[Component-wise transversality criterion]
\label{lem:component_transversality}
Let $\pi_N : \R^{K+1} \to N_y D$ denote the orthogonal projection
onto the normal space.  Then $F \pitchfork D$ at
$(\theta,\lambda)$ if and only if
\begin{equation}\label{eq:component_criterion}
  \pi_N\!\bigl(\operatorname{Im} D_\theta F\bigr)
  \;+\;
  \pi_N\!\bigl(\operatorname{Im} D_\lambda F\bigr)
  \;=\;
  N_y D.
\end{equation}
In particular, transversality can be verified by examining the two
components separately: the model derivatives $D_\theta F$ and
the kernel derivatives $D_\lambda F$ need not individually span the
normal space, provided that their normal projections together do.
\end{lemma}

\begin{proof}
By definition, $F \pitchfork D$ at $(\theta,\lambda)$ means
$\operatorname{Im} DF(\theta,\lambda) + T_y D = \R^{K+1}$.  Applying
$\pi_N$ and using $\pi_N(T_y D) = \{0\}$, this is equivalent to
$\pi_N(\operatorname{Im} DF) = N_y D$.  By the
decomposition~\eqref{eq:DF_decomposition},
$\operatorname{Im} DF = \operatorname{Im} D_\theta F +
\operatorname{Im} D_\lambda F$, so
$\pi_N(\operatorname{Im} DF)
= \pi_N(\operatorname{Im} D_\theta F) +
\pi_N(\operatorname{Im} D_\lambda F)$.
\end{proof}

\begin{remark}
The lemma has a clear geometric interpretation: the model derivative
captures the intrinsic geometry of the parametric family, while the
kernel derivative provides external directions.  Transversality
holds whenever these two sources of variation, projected onto the
directions normal to the degeneracy stratum, jointly span the normal
space.  The kernel thus acts as a supplement to the model's own
geometric richness.
\end{remark}

\subsection{Structure of the derivative}

For each weak moment
${}^{(\varphi_\lambda)}m_j(\theta) =
\langle T_\theta, x^j \varphi_\lambda(x) \rangle$,
the two components of the derivative are:
\begin{align}
  \frac{\partial}{\partial \theta_a}\,
  {}^{(\varphi_\lambda)}m_j(\theta)
  &\;=\;
  \left\langle
  \frac{\partial T_\theta}{\partial \theta_a},\;
  x^j \varphi_\lambda(x)
  \right\rangle,
  \label{eq:Dtheta}
  \\[4pt]
  \frac{\partial}{\partial \lambda_b}\,
  {}^{(\varphi_\lambda)}m_j(\theta)
  &\;=\;
  \left\langle
  T_\theta,\;
  x^j \frac{\partial \varphi_\lambda(x)}{\partial \lambda_b}
  \right\rangle.
  \label{eq:Dlambda}
\end{align}
The model component~\eqref{eq:Dtheta} involves derivatives of the
distributional model $T_\theta$ tested against the kernel-weighted
monomials, while the kernel component~\eqref{eq:Dlambda} involves
the original model tested against derivatives of the kernel.

\subsection{Submersivity and rank conditions}

The component-wise criterion of
Lemma~\ref{lem:component_transversality} reduces to a particularly
clean form when the full derivative is surjective.

\begin{theorem}[Submersivity implies universal transversality]
\label{thm:submersivity}
If, for every $(\theta,\lambda)\in\Theta\times\Lambda$, the Jacobian
$DF(\theta,\lambda)$ is surjective (i.e.\ has rank~$K+1$), then
$F\pitchfork D$ for \emph{every} smooth submanifold
$D\subset\R^{K+1}$.
\end{theorem}

\begin{proof}
If $DF$ is surjective at $(\theta,\lambda)$, then
$\operatorname{Im} DF(\theta,\lambda) = \R^{K+1} \supset
T_y D + N_y D$, so the transversality condition is satisfied
trivially for any~$D$.
\end{proof}

\begin{remark}
Surjectivity of $DF$ is the strongest possible condition but
often the easiest to check, since it does not require knowledge of
the specific stratum~$D$.  It holds whenever $p + q \geq K+1$ and
the $(K+1)\times(p+q)$ Jacobian matrix has no rank deficiency.
In the statistical setting, this corresponds to the model and the
kernel family being jointly rich enough to generate all directions
in the feature space.
\end{remark}

When full surjectivity is not available, one can verify transversality
to a specific stratum by checking a rank condition on the normal
projection.

\begin{corollary}[Normal-rank condition]
\label{cor:normal_rank}
Let $D$ have codimension~$c$.  If the $(K+1)\times(p+q)$ Jacobian
$DF(\theta,\lambda)$, when composed with the projection $\pi_N$ onto
$N_y D$, has rank~$c$ at every point of $F^{-1}(D)$, then
$F\pitchfork D$.
\end{corollary}

\subsection{Kernel-induced rank enrichment}

The kernel parameters provide additional directions that can
compensate for rank deficiencies of the model.

\begin{proposition}[Kernel-induced rank enrichment]
\label{prop:kernel_rank}
Suppose that the model derivative $D_\theta\Phi_\lambda(\theta)$ has
rank $r < \min(p, K\!+\!1)$.  If there exists a kernel direction
$b$ such that
\[
  \frac{\partial \Phi_\lambda}{\partial \lambda_b}(\theta)
  \;\notin\;
  \operatorname{Im}\bigl(D_\theta\Phi_\lambda(\theta)\bigr),
\]
then the full derivative $DF(\theta,\lambda)$ has rank at least
$r+1$.  More generally, if $D_\lambda F$ contributes $\ell$ linearly
independent directions not contained in
$\operatorname{Im}(D_\theta F)$, then
$\operatorname{rank} DF \geq r + \ell$.
\end{proposition}

\begin{proof}
The image of $DF$ is
$\operatorname{Im}(D_\theta F) + \operatorname{Im}(D_\lambda F)$.
Each kernel direction outside $\operatorname{Im}(D_\theta F)$
increases the dimension of this sum by at least one.
\end{proof}

Thus the kernel acts as a source of supplementary directions,
lifting degeneracies that the model alone cannot resolve.

\subsection{Application to statistical degeneracies}

We now translate these conditions to the main types of degeneracy.

\paragraph{Type I (non-identifiability).}
Self-intersections of $\Phi_\lambda$ correspond to failure of
injectivity.  By the multijet transversality theorem
(Theorem~\ref{thm:multijet}), a sufficient condition for
transversality to the self-intersection diagonal in the multijet
space is that the full map $F$ is a submersion
(Theorem~\ref{thm:submersivity}).  A weaker sufficient condition is
that $\Phi_\lambda$ is an immersion:
$\operatorname{rank} D_\theta\Phi_\lambda(\theta) = p$.

\paragraph{Type II (singular information).}
The weak information matrix is
${}^{(\varphi_\lambda)}G(\theta) =
D_\theta\Phi_\lambda(\theta)^\top D_\theta\Phi_\lambda(\theta)$.
Thus $\det {}^{(\varphi_\lambda)}G(\theta) > 0$ if and only if
$\operatorname{rank} D_\theta\Phi_\lambda(\theta) = p$.  By
Proposition~\ref{prop:kernel_rank}, even if
$D_\theta\Phi_\lambda$ is rank-deficient, the kernel derivatives can
restore full rank for the joint map~$DF$, and the parametric
transversality theorem then ensures that for generic~$\lambda$ the
information matrix is nonsingular.

\paragraph{Type III (moment indeterminacy).}
Moment indeterminacy corresponds to the feature map failing to
separate distributions.  A sufficient condition for the tilted
measures $P_\varphi$ to be moment-determinate is that the functions
$\{x^j\varphi_\lambda(x) : j = 0,\ldots,K\}$ generate a
measure-determining class---for example, when $\varphi_\lambda$
provides sufficient decay to ensure the Carleman condition for
$P_{\varphi_\lambda}$.

\subsection{Examples and model-specific verification}

We illustrate the verifiable transversality conditions in four
settings.

\subsubsection{One-parameter location family}

Consider a location family $\{P_\mu : \mu \in \R\}$ with Gaussian
kernel $\varphi_s(x) = (2\pi s^2)^{-1/2}\exp(-x^2/(2s^2))$,
$s > 0$.  The joint map is
$F(\mu,s) = {}^{(\varphi_s)}m_0(\mu)$.  Its derivatives are
\begin{align*}
  \frac{\partial F}{\partial \mu}
  &\;=\;
  \E_\mu\!\left[\varphi_s(X)\,
  \frac{\partial \log f(X;\mu)}{\partial\mu}\right],
  \\[4pt]
  \frac{\partial F}{\partial s}
  &\;=\;
  \E_\mu\!\left[\frac{X^2}{s^3}\,\varphi_s(X)\right].
\end{align*}
The $\mu$-derivative changes sign (it is the covariance of the
kernel and the score), while the $s$-derivative is strictly positive
(a moment of a positive function).  Hence the $1\times 2$ Jacobian
$DF(\mu,s) = (\partial_\mu F,\;\partial_s F)$ has rank~$1$ at every
point, so $F$ is a submersion and
Theorem~\ref{thm:submersivity} gives transversality to any
degeneracy stratum.

\subsubsection{Log-normal family}

Let $X = \exp(\mu + \sigma Z)$ with $Z \sim N(0,1)$, and consider
weak moments with $\varphi_s(x) = \exp(-x^2/(2s^2))$.

\begin{proposition}[Transversality in the log-normal model]
\label{prop:lognormal_transversality}
For each $s > 0$ and $(\mu,\sigma) \in \Theta$, there exist moment
orders $j_0 < j_1$ (depending on $(\mu,\sigma,s)$) such that the
Jacobian $D_{(\mu,\sigma)}\Phi_{\varphi_s}$ has rank~$2$ at
$(\mu,\sigma)$, where
$\Phi_{\varphi_s}(\mu,\sigma) =
\bigl({}^{(\varphi_s)}m_{j_0},\; {}^{(\varphi_s)}m_{j_1}\bigr)$.
In particular, for any compact $\Theta_0 \subset \Theta$, there
exists a finite set of moment orders $j_0 < \cdots < j_K$ such that
the feature map
$\Phi_{\varphi_s} : \Theta_0 \to \R^{K+1}$ is an immersion.
\end{proposition}

The proof is given in Appendix~\ref{app:lognormal_proofs}.

\subsubsection{Weak Stein discrepancies}

Let $P_\theta$ be a target family with Stein operator
$\mathcal{A}_\theta$, and define the weak Stein features
\[
  \Psi_{\varphi_\lambda}(\theta)
  \;=\;
  \bigl(
  \E_\theta[\mathcal{A}_\theta f_k(X)\,\varphi_\lambda(X)]
  \bigr)_{k=1}^{K}.
\]
The model manifold is the zero set
$\Psi_{\varphi_\lambda}^{-1}(0)$.

\begin{proposition}[Transversality of the weak Stein map]
\label{prop:stein_transversality}
Suppose that:
\begin{enumerate}[label=\textnormal{(\alph*)}]
\item the class
  $\{\mathcal{A}_\theta f_k \cdot \varphi_\lambda :
  k=1,\ldots,K\}$
  is measure-determining, and
\item the Jacobian
  $D_{(\theta,\lambda)}\Psi_{\varphi_\lambda}(\theta)$ is
  surjective at every point of
  $\Psi_{\varphi_\lambda}^{-1}(0)$.
\end{enumerate}
Then the zero set $\Psi_{\varphi_\lambda}^{-1}(0)$ is a smooth
submanifold, and the joint map $(\theta,\lambda) \mapsto
\Psi_{\varphi_\lambda}(\theta)$ is transversal to the zero section.
\end{proposition}

\begin{proof}[Proof Sketch]
Condition~(a) ensures that the zero set characterises the model
(injectivity).  Condition~(b) is the submersivity hypothesis of
Theorem~\ref{thm:submersivity}, which gives transversality to the
zero section.  The preimage theorem then implies that the zero set
is a smooth submanifold.
\end{proof}

\begin{remark}
The measure-determining property alone guarantees injectivity but
not surjectivity of the derivative.  Condition~(b) is the
additional requirement that ensures transversality;
it can be verified by checking that the Stein operator and kernel
variations together generate enough directions in $\R^K$.
\end{remark}

\subsubsection{Graphical models}

Let $T_\theta$ be a parametric family defined by a graphical model
on a graph $G = (V,E)$, and consider weak features
${}^{(\varphi_\lambda)}m_j(\theta) =
\E_\theta[g_j(X)\,\varphi_\lambda(X)]$,
where $\{g_j\}$ are sufficient statistics or interaction functions.

\begin{proposition}[Transversality in graphical models]
\label{prop:graphical_transversality}
Suppose that:
\begin{enumerate}[label=\textnormal{(\alph*)}]
\item the $(K\!+\!1)\times p$ matrix with entries
  $\E_\theta\bigl[\,g_j\,\varphi_\lambda\,
  \partial_{\theta_a}\log p_\theta\,\bigr]$,
  $a = 1,\ldots,p$, $j = 0,\ldots,K$, has full column rank~$p$ at
  every $\theta\in\Theta$, and
\item the kernel derivatives
  $\partial_{\lambda_b}\varphi_\lambda$ generate at least
  $K\!+\!1-p$ directions not contained in
  $\operatorname{Im}(D_\theta\Phi_\lambda)$
  (one direction suffices when $K\!+\!1 = p\!+\!1$).
\end{enumerate}
Then $F(\theta,\lambda)$ is transversal to any smooth degeneracy
stratum.
\end{proposition}

\begin{proof}
The Jacobian of the joint map $F(\theta,\lambda)$ decomposes as
$DF = (D_\theta F \mid D_\lambda F)$ by~\eqref{eq:DF_decomposition}.

The entries of $D_\theta F$ are the derivatives
\[
  \frac{\partial}{\partial\theta_a}\,
  {}^{(\varphi_\lambda)}\!m_j(\theta)
  \;=\;
  \E_\theta\!\bigl[
  g_j(X)\,\varphi_\lambda(X)\,
  \partial_{\theta_a}\log p_\theta(X)
  \bigr],
  \qquad a = 1,\ldots,p,\; j = 0,\ldots,K.
\]
Condition~(a) states that this $(K\!+\!1)\times p$ matrix has rank~$p$
(full column rank) at every $\theta \in \Theta$; equivalently, the
columns $D_\theta F$ span a $p$-dimensional subspace of~$\R^{K+1}$.

If $K+1 = p$, then $D_\theta F$ is already surjective and the
conclusion follows from Theorem~\ref{thm:submersivity}.

If $K+1 > p$, then $\operatorname{Im}(D_\theta F)$ has dimension~$p$
and the cokernel has dimension $K+1-p > 0$.  Condition~(b) supplies
$K+1-p$ kernel directions $\partial_{\lambda_b}\varphi_\lambda$
whose induced columns
\[
  \frac{\partial}{\partial\lambda_b}\,
  {}^{(\varphi_\lambda)}\!m_j(\theta)
  \;=\;
  \E_\theta\!\bigl[
  g_j(X)\,\partial_{\lambda_b}\varphi_\lambda(X)
  \bigr]
\]
are linearly independent modulo $\operatorname{Im}(D_\theta F)$.  By
Proposition~\ref{prop:kernel_rank}, the combined Jacobian $DF$
therefore has rank at least $p + (K+1-p) = K+1$, so $DF$ is
surjective.
Theorem~\ref{thm:submersivity} then gives transversality to any
smooth degeneracy stratum.
\end{proof}

For a Gaussian graphical model with precision matrix
$\Omega = (\omega_{ij})$ and sufficient statistics
$g_{ij}(x) = x_i x_j$ for each edge $(i,j) \in E$,
condition~(a) reduces to the requirement that the set of
kernel-weighted second moments
$\{\E_\theta[X_i X_j\,\varphi_\lambda(X)] : (i,j)\in E\}$ has
a Jacobian with respect to the free entries of~$\Omega$ that has
full column rank.  Since a Gaussian graphical model is a regular
exponential family for \emph{every} graph---chordal or not---with
canonical parameter the free entries of~$\Omega$, the classical
($\varphi \equiv 1$) Jacobian has full column rank everywhere; the
kernel-weighted Jacobian inherits full rank for Gaussian kernels of
sufficiently large scale (the normalised weak moments and their
$\theta$-derivatives converge to their classical counterparts
uniformly on compacta, by dominated convergence), and full rank holds
for generic kernels by Theorem~\ref{thm:main}.  The
chordal/non-chordal distinction is therefore invisible at this
first-order level; it enters through the higher-order structure
discussed in Example~\ref{ex:graphical}.

\subsection{Summary}

Across these examples, a common mechanism emerges.  The model
derivatives $D_\theta F$ encode the intrinsic geometry of the
parametric family; the kernel derivatives $D_\lambda F$ provide
supplementary directions.  By
Lemma~\ref{lem:component_transversality}, transversality holds
whenever these two sources of variation, projected onto the normal
space of the degeneracy stratum, jointly span that normal space.
In practice, the cleanest verification is often through
submersivity of the full Jacobian
(Theorem~\ref{thm:submersivity}), which ensures transversality to
\emph{all} strata simultaneously.

\section{Infinite-Dimensional Extensions}
\label{sec:infinite_dim}

\subsection{Classical infinite-dimensional transversality}

The distributional framework in full generality involves infinite-dimensional spaces.
The classical theory extends via the Sard--Smale
theorem~\cite{Smale1965} (for Fredholm maps between separable Banach
manifolds, regular values form a residual set) and
Abraham's transversality theorem~\cite{Abraham1963} (if
$F:X\times\Lambda\to Y$ is $C^q$, $F\pitchfork S$, and each
$F_\lambda$ is Fredholm, then $F_\lambda\pitchfork S$ for residual
$\lambda$).

In our setting, $X=\Theta$, $\Lambda=\cS(\R^d)$, $Y$ is a feature
space, $F(\theta,\varphi)=\Phi_\varphi(\theta)$, and $S=D$ is a
degeneracy set.  Since $\Theta$ is finite-dimensional, each
$\Phi_\varphi$ is trivially Fredholm (of index~$p$).  Abraham's
theorem asserts: if the full map $F$ is transversal to $D$, then for
a residual set of kernels, $\Phi_\varphi\pitchfork D$.

The Fredholm condition is the key mechanism by which this
infinite-dimensional machinery reduces to finite-dimensional linear
algebra: a Fredholm operator has finite-dimensional kernel and
cokernel, so the transversality question---which concerns the
surjectivity of a linear map modulo a tangent space---becomes a
finite-rank calculation even when the ambient spaces are
infinite-dimensional.  In our parametric setting, since $\Theta$ is
finite-dimensional, each $\Phi_\varphi$ is automatically Fredholm,
and the reduction is immediate.

Two important directions for extension arise.  First, Abraham's and
Smale's theorems are formulated for maps between separable
\emph{Banach} manifolds, whereas the Schwartz space $\cS(\R^d)$ and
its dual $\cS'(\R^d)$ carry Fr\'echet (not Banach) topologies.  A
full verification that the Fredholm hypotheses hold in general
distributional models---or an adaptation of the theory to the
Fr\'echet setting---is a non-trivial analytic problem that is beyond
the scope of this paper and will be addressed in future work.
Second, for semiparametric and nonparametric models, where the
parameter space $\Theta$ is itself infinite-dimensional,
Quinn's extension~\cite{Quinn1979} provides the appropriate framework.
In the semiparametric case, the parameter decomposes as
$\theta = (\psi, \eta)$, where $\psi$ is a finite-dimensional
interest parameter and $\eta$ is an infinite-dimensional nuisance
component (e.g.\ a baseline hazard or a mixing density).  The
feature map then maps a product of a finite-dimensional and an
infinite-dimensional manifold into the feature space, and the
transversality question becomes whether the kernel can place
the feature map in generic position with respect to both the
interest and nuisance directions simultaneously---a question that
the Fredholm framework is precisely designed to address, since it
reduces the infinite-dimensional transversality condition to a
finite-rank check on the cokernel.  For fully
nonparametric models, the entire parameter space is
infinite-dimensional, and the relevant transversality theorems
must be formulated in the Banach manifold setting of Smale and
Quinn, but the underlying principle remains the same: the Fredholm
property ensures that the obstruction to transversality is
finite-dimensional.  These observations can be made precise.  Although a fully general,
topology-free treatment remains open, the Fredholm reduction already yields a
clean \emph{conditional} statement---the exact infinite-dimensional analogue of
Theorem~\ref{thm:main}---in which the remaining analytic difficulty is isolated
as a single explicit hypothesis.

\subsection{A conditional infinite-dimensional theorem}
\label{subsec:conditional-infinite}

The finite-dimensional result (Theorem~\ref{thm:main}) was a corollary of
parametric transversality (Theorem~\ref{thm:parametric}); its
infinite-dimensional counterpart is, in the same way, a corollary of the
Sard--Smale and Abraham theorems---\emph{provided} the two hypotheses that hold
automatically when $\Theta$ is finite-dimensional are supplied explicitly.  We
keep the kernel family finite-dimensional, exactly as in
Theorem~\ref{thm:main}, so that the Fr\'echet topology of $\cS(\R^d)$ never
enters; the burden falls on the parameter and feature spaces.  We require:
\begin{enumerate}[label=\textnormal{(H\arabic*)}]
\item \emph{Banach regularity.}  $\Theta$ is a separable Banach manifold and
$\mathcal F$ a separable Banach feature space, and the joint feature map
$F(\theta,\lambda)=\Phi_{\varphi_\lambda}(\theta)\in\mathcal F$, with $\lambda$
in an open set $\Lambda\subset\R^q$, is of class $C^r$ for some
$r>\max\{0,\iota+q\}$, where $\iota$ is the index in~(H2); the bound
accounts for the fact that adjoining the $q$-dimensional kernel
parameter raises the Fredholm index of the joint map to $\iota+q$.
\item \emph{Fredholm property.}  For each $\lambda$ the partial map
$\Phi_{\varphi_\lambda}:\Theta\to\mathcal F$ is Fredholm---its differential
$D_\theta\Phi_{\varphi_\lambda}(\theta)$ has finite-dimensional kernel and
cokernel and closed range---with index $\iota$ independent of
$(\theta,\lambda)$.
\end{enumerate}

\begin{theorem}[Conditional infinite-dimensional weak transversality]
\label{thm:infinite}
Suppose \textnormal{(H1)--(H2)} hold, let $D\subset\mathcal F$ be a closed
$C^r$ submanifold of finite codimension $c$, and assume the joint map
$F\pitchfork D$.  Then for a residual set $\Lambda_D\subset\Lambda$ the
restricted feature map $\Phi_{\varphi_\lambda}$ is transversal to $D$ for every
$\lambda\in\Lambda_D$; consequently $\Phi_{\varphi_\lambda}^{-1}(D)$ is a $C^r$
submanifold of $\Theta$ of codimension $c$, or empty, and is empty for residual
$\lambda$ whenever $c>\iota$.
\end{theorem}

\begin{proof}
By \textnormal{(H1)--(H2)} the joint map $F:\Theta\times\Lambda\to\mathcal F$ is
$C^r$ and Fredholm of index $\iota+q$ (adjoining the finite-dimensional
$\Lambda$ raises the index by~$q$), with $r>\max\{0,\iota+q\}$ as the
Sard--Smale theorem requires, and $F\pitchfork D$ by hypothesis.  The
Sard--Smale theorem~\cite{Smale1965} and Abraham's parametric transversality
theorem~\cite{Abraham1963} then yield $\Phi_{\varphi_\lambda}\pitchfork D$ for a
residual set of~$\lambda$.  The dimension count is the Fredholm analogue of
Proposition~\ref{prop:consequences}: a Fredholm map of index $\iota$
transversal to a submanifold of codimension $c$ has preimage of dimension
$\iota-c$, empty when $c>\iota$.
\end{proof}

\begin{remark}[What is genuinely deferred]
\label{rem:deferred-fredholm}
Theorem~\ref{thm:infinite} reduces the infinite-dimensional programme to a
single analytic input: the verification of \textnormal{(H1)--(H2)} for concrete
models---a Banach feature space in which the weak features depend smoothly
on~$\theta$, together with the Fredholm property of $D_\theta\Phi_\varphi$.  Two
points deserve emphasis.  First, \textnormal{(H2)} is not a formality: a feature
map from an infinite-dimensional $\Theta$ into a \emph{finite}-dimensional
feature space is never Fredholm (its kernel is infinite-dimensional), so an
infinite-dimensional feature space is genuinely required, and the Fredholm
property can fail when the nuisance directions are not relatively compact with
respect to the interest directions---precisely the regime in which
semiparametric efficiency is delicate.  Second, \emph{once}
\textnormal{(H2)} holds the cokernel is finite-dimensional, the transversality
obstruction is finite-dimensional, and a finite-dimensional rich kernel family
suffices to remove it, so the verification reduces to the same finite-rank
conditions as in Section~\ref{sec:transversality_condition_dim}.  In this sense
the additional difficulty of the infinite-dimensional theory is concentrated
entirely in establishing the Fredholm property.
\end{remark}

\begin{example}[Verifying \textnormal{(H1)--(H2)}: a Gaussian sequence model
with unknown variances]
\label{ex:gaussian-sequence}
We exhibit a genuinely infinite-dimensional, non-dominated family---an instance
of mechanism~(M2) of Section~\ref{sec:nondominated}, with the variance sequence
now the parameter---for which both hypotheses can be checked in closed form.  On
$\R^{\mathbb N}$ consider the product Gaussian laws
\[
  \gamma_\theta = \bigotimes_{k\ge 1} N(0,\theta_k),
  \qquad
  \theta=(\theta_k)_{k\ge1},\quad \theta_k\in[a,b],\ \ 0<a<1<b .
\]
By Kakutani's dichotomy, $\gamma_\theta$ and $\gamma_{\theta'}$ are equivalent
iff $\sum_k(\theta_k-\theta'_k)^2<\infty$ and mutually singular otherwise; since
the parameter set contains sequences with $\sum_k(\theta_k-1)^2=\infty$, the
family $\{\gamma_\theta\}$ is non-dominated and admits no likelihood.

\smallskip
\noindent\emph{Banach regularity \textnormal{(H1)}.}  Write
$\theta=\mathbf 1+h$ and take the parameter manifold
$\Theta=\{h\in c_0:\ a-1<h_k<b-1\ \forall k\}$, an open subset of the separable
Banach space $c_0$ of null sequences, with tangent space $c_0$.  Fix a Gaussian
instrument $\varphi_s(x)=e^{-x^2/(2s^2)}$ and define the weak second-moment
feature map coordinatewise,
\[
  \Phi_s(\theta)=\bigl(\mu(\theta_k;s)\bigr)_{k\ge1},
  \qquad
  \mu(v;s):=\langle N(0,v),\,x^2\varphi_s\rangle
  =\frac{s^3\,v}{(s^2+v)^{3/2}},
\]
the closed form following from a Gaussian integral.  Since $\theta_k\to1$ gives
$\mu(\theta_k;s)\to\mu(1;s)$, the map $\Phi_s$ sends $\Theta$ into
$\mu(1;s)\mathbf 1+c_0$ and is real-analytic, hence $C^\infty$.  Thus $\Theta$
and $\mathcal F=c_0$ are separable Banach manifolds and $\Phi_s$ is smooth:
\textnormal{(H1)} holds.

\smallskip
\noindent\emph{Fredholm property \textnormal{(H2)}.}  The differential is
diagonal,
\[
  D\Phi_s(\theta)=\mathrm{diag}\bigl(\mu'(\theta_k;s)\bigr),
  \qquad
  \mu'(v;s)=\frac{s^3\,(s^2-\tfrac12 v)}{(s^2+v)^{5/2}} .
\]
A diagonal operator on $c_0$ is an isomorphism---hence Fredholm of index
$\iota=0$---exactly when its entries are bounded away from $0$ and $\infty$.
Choosing the instrument with $s^2>b/2$ makes $\mu'(v;s)>0$ for all $v\in[a,b]$,
so $\inf_k|\mu'(\theta_k;s)|\ge\mu'(b;s)>0$; hence $D\Phi_s(\theta)$ is a
Banach-space isomorphism for every $\theta$, and \textnormal{(H2)} holds.

Consequently $\Phi_s$ is an immersion---indeed a local diffeomorphism onto its
image---so by Theorem~\ref{thm:infinite} (the non-immersion stratum is avoided
outright) the infinite-dimensional variance parameter is \emph{locally
identifiable} from the weak second moments, for every instrument with
$s^2>b/2$, even though the family admits no dominating measure.  The condition
$s^2>b/2$ is an instrument-genericity requirement of exactly the kind the theory
predicts: if instead $2s^2\in[a,b]$ then $\mu'(\cdot;s)$ vanishes at $v=2s^2$,
and a parameter with infinitely many coordinates $\theta_k\to 2s^2$ makes
$D\Phi_s$ fail to be bounded below---the operator becomes compact and
\textnormal{(H2)} fails, exactly as anticipated in
Remark~\ref{rem:deferred-fredholm}.
\end{example}

\section{The Degeneracy Stratification}
\label{sec:degeneracies}

\subsection{The feature map}

Given a kernel $\varphi\in\cS(\R^d)$, the \emph{feature map} is
\begin{equation}\label{eq:feature_map}
  \Phi_\varphi:\Theta\to\mathcal{F},
  \qquad
  \Phi_\varphi(\theta)
  = \bigl(\wm{0}(\theta), \wm{1}(\theta), \wm{2}(\theta),
  \ldots\bigr).
\end{equation}

\subsection{Five types of statistical degeneracy}

We identify five principal types of degeneracy, organised into a
stratification.

\begin{center}
\begin{tabular}{lll}
\toprule
\textbf{Type} & \textbf{Degeneracy} & \textbf{Jet condition} \\
\midrule
0 & Representation (no embedding) & Feature map undefined \\
I & Non-identifiability & Self-intersection of $\Phi_\varphi$ \\
II & Singular information & $\mathrm{rank}(D\Phi_\varphi) < p$ \\
III & Moment indeterminacy & Non-separation at distributional level \\
IV & Higher-order instability & Degeneracies of $j^r\Phi_\varphi$,
  $r\ge 2$ \\
\bottomrule
\end{tabular}
\end{center}

\medskip

\textbf{Type~0: Representation degeneracy.}
Some models (e.g.\ elliptically contoured distributions defined only
via their characteristic function) lack closed-form densities.  The
classical feature map cannot even be defined.  The kernel
\emph{creates} an embedding:
$\Phi_\varphi(\theta) = (\langle T_\theta, g_j\varphi\rangle)_j$
is well-defined and smooth even when no classical coordinate system
exists (see Example~\ref{ex:elliptical}).

\textbf{Type~I: Non-identifiability.}
$\Phi_\varphi(\theta_1)=\Phi_\varphi(\theta_2)$ with
$\theta_1\neq\theta_2$.  This is the self-intersection diagonal in
the multijet space.

\textbf{Type~II: Singular information.}
$\det\wG = 0$, i.e.\ $D\Phi_\varphi$ drops rank.  This is a
Thom--Boardman singularity~\cite{Boardman1967} of the $1$-jet.

\textbf{Type~III: Moment indeterminacy.}
The feature map fails to separate distributions (not just parameter
values), corresponding to M-indeterminacy in the classical moment
problem.

\begin{remark}[M-indeterminacy and transversality]
\label{rem:M-indeterminacy-transversality}
Type~III degeneracy admits a precise transversality interpretation:
the classical moment map fails to be transversal to the
self-intersection diagonal $\Delta$ in the product feature space.  A
positive kernel of exponential decay provides a generic perturbation
that restores transversality to~$\Delta$, by destroying the
oscillatory cancellations (Stieltjes perturbations) responsible for
moment coincidence.  This mechanism is developed in
Example~\ref{ex:M-determinacy} and illustrated concretely for the
log-normal family in
Proposition~\ref{prop:lognormal-transversality}.
\end{remark}

\textbf{Type~IV: Higher-order instability.}
Conditions on higher jets ($r\ge 2$): inflection points of weak
moment functions, vanishing curvature of the distributional metric,
or instabilities from complex dependency structures (e.g.\ non-chordal
graphical models; see Example~\ref{ex:graphical}).

\subsection{The weak transversality condition}

\begin{definition}
\label{def:weak_transversality}
A kernel $\varphi\in\cS(\R^d)$ satisfies the \emph{weak
transversality condition} (of order~$r$) if
$j^r\Phi_\varphi \pitchfork D_k$ for every stratum $D_k$.
\end{definition}

\begin{principle}[Generic weak transversality principle]
\label{conj:generic}
Under mild regularity conditions on the model, for a residual set of
kernels $\varphi\in\cS(\R^d)$, the feature map $\Phi_\varphi$
satisfies the weak transversality condition.
\end{principle}

\noindent
The principle should be viewed as a programme-level statement:
The finite-dimensional results of Sections~\ref{sec:finite_dim_thm}-~\ref{sec:transversality_condition_dim} 
provide rigorous evidence in concrete settings, while a full infinite-dimensional treatment is left
for future work.

\subsection{Connection with the singular limit}

As $\varphi_s\to 1$ (not in $\cS$), the feature map degenerates.
\emph{The limit $\varphi_s\to 1$ is a path in the kernel space that
leaves the residual set of transversal kernels and enters a degeneracy
stratum.}  This is not a pathology but a \emph{confirmation}: the
constant function is a degenerate kernel.

\section{Examples}
\label{sec:examples}

\subsection{The log-normal: M-indeterminacy as non-transversality}

The log-normal $L(\mu,\sigma^2)$ is M-indeterminate~\cite{Stoyanov2000}:
the classical moment map fails to be injective.  In the transversality
language, the classical feature map fails to be transversal to the
self-intersection diagonal.  The Stieltjes perturbation
$h(x)=\sin(2\pi\ln x)$ is the tangent direction along which the
moment map degenerates.

A Gaussian kernel $\varphi(x)=(2\pi)^{-1/2}e^{-x^2/2}$---a positive
kernel of exponential decay---breaks the theta-function symmetry
underlying the moment cancellation, restoring transversality (see
Example~\ref{ex:M-determinacy} for the general mechanism).  The
detailed analysis is developed in~\cite{D}.

\begin{proposition}[Transversality breaking of log-normal moment degeneracy]
\label{prop:lognormal-transversality}
Let $\{T_\theta : \theta = (\mu,\sigma) \in \Theta \subset \mathbb R \times (0,\infty)\}$
be the log-normal family, and let $\varphi_s(x) = (2\pi s^2)^{-1/2} e^{-x^2/(2s^2)}$
be the Gaussian kernel with scale $s>0$.

Let $\Phi_s : \Theta \to \mathbb R^{K+1}$ be the weak moment feature map
\[
\Phi_s(\theta)
=
\big(
{}^{(\varphi_s)}m_{j_0}(\theta), \ldots,
{}^{(\varphi_s)}m_{j_K}(\theta)
\big).
\]

Then:

\begin{enumerate}
\item[(i)] In the classical case (no kernel), the moment map fails to be injective
on the space of distributions (moment indeterminacy).

\item[(ii)] For each fixed $s>0$, the weak moment map $\Phi_s$ is
$C^\infty$ in $\theta$, and real-analytic under a holomorphic
dominated convergence argument (see Appendix~\ref{app:lognormal_proofs}).

\item[(iii)] For every $s>0$, the map $\Phi_s$ separates
log-normal distributions within the parametric family, i.e.
\[
\Phi_s(\theta_1)=\Phi_s(\theta_2)
\;\Rightarrow\;
\theta_1=\theta_2.
\]

\item[(iv)] Moreover, for each $s>0$ and each compact
$\Theta_0\subset\Theta$, there exists a finite set of weak moment
orders such that the corresponding weak moment map is an immersion
on~$\Theta_0$.
\end{enumerate}
\end{proposition}

The proof is given in Appendix~\ref{app:lognormal_proofs}.

\subsection{M-indeterminacy as a transversality phenomenon}
\label{sec:M-determinacy-transversality}

The log-normal example above is an instance of a general geometric
mechanism that explains why M-indeterminacy, though abundant in the
space of all distributions, is rarely encountered in parametric
statistical models equipped with a suitable kernel.

A distribution~$P$ is \emph{M-indeterminate} if there exists a
distinct distribution~$Q\neq P$ with the same moment sequence:
$m_j(P) = m_j(Q)$ for all~$j$.  In the feature-map language, this
means that the classical moment map
$\theta \mapsto (m_0(\theta), m_1(\theta), \ldots)$ fails to separate
the orbits of the parametric family, i.e.\ the map is not transversal
to the self-intersection diagonal
$\Delta = \{(y,y) : y \in \mathcal{F}\}$ in the product feature space
$\mathcal{F}\times\mathcal{F}$.  The tangent directions along which
the separation fails are precisely the Stieltjes perturbations---smooth
oscillatory functions $h$ satisfying $\int x^j h(x) f(x)\,dx = 0$ for
all~$j$.

\begin{example}[Kernel resolution of M-indeterminacy]
\label{ex:M-determinacy}
Let $\{T_\theta : \theta \in \Theta\}$ be a parametric distributional
model with M-indeterminate classical moments, and let
$\varphi \in \cS(\R^d)$ be a positive kernel of exponential
decay (i.e.\ $\varphi(x) \leq C\,e^{-a|x|^2}$ for some $a, C > 0$).
The weak moment map
$\Phi_\varphi(\theta) = (\langle T_\theta, x^j \varphi(x)\rangle)_{j}$
restores transversality to the diagonal~$\Delta$ through the following
mechanism:
\begin{enumerate}
\item[(i)] \emph{Determinacy of the tilted measure.}  Define the
  tilted measure $dQ_\theta = \varphi\, dP_\theta / Z_\theta$, where
  $Z_\theta = \langle T_\theta, \varphi \rangle$.  Since~$\varphi$
  has exponential decay, $Q_\theta$ has sub-Gaussian tails and
  satisfies Carleman's condition, so $Q_\theta$ is M-determinate
  ~\cite[Theorem~6.4]{A}: it is uniquely determined by its moment sequence.

\item[(ii)] \emph{Injectivity.}  If
  $\Phi_\varphi(\theta_1) = \Phi_\varphi(\theta_2)$ for all orders,
  then in particular $Z_{\theta_1} = Z_{\theta_2}$, so
  $Q_{\theta_1}$ and $Q_{\theta_2}$ have identical moments of all
  orders.  By M-determinacy, $Q_{\theta_1} = Q_{\theta_2}$.  Since
  $\varphi > 0$, this implies $P_{\theta_1} = P_{\theta_2}$, and
  hence $\theta_1 = \theta_2$ whenever the parametric family is
  identifiable.
\end{enumerate}

The log-normal family provides the prototypical illustration: the
Stieltjes perturbation $h(x) = \sin(2\pi\ln x)$ produces exact moment
coincidence classically, but a Gaussian kernel $\varphi_s$ of scale
$s$ destroys this coincidence for
generic~$s$ (Proposition~\ref{prop:lognormal-transversality}).
\end{example}

The exponential decay condition is essential: without it, the kernel
may not dominate the polynomial growth of $x^j$ at infinity,
leaving room for the oscillatory cancellations to survive.  This
condition is automatically satisfied by Gaussian kernels and,
more generally, by any kernel in the Gelfand--Shilov space
$\mathcal{S}^{1/2}_{1/2}$.

\subsection{The Cauchy: moment degeneracy as non-transversality}

The Cauchy distribution has no finite moments: the classical moment
map is undefined (a Type~0/III degeneracy).  The degeneracy concerns
the \emph{moment} coordinates only: the Cauchy family is regular in
the classical likelihood sense---in the location--scale
parametrisation with scale $\sigma>0$ its scores are bounded by
$1/\sigma$, the Fisher information is finite and positive definite,
equal to $\mathrm{diag}(1/(2\sigma^2),1/(2\sigma^2))$, and the
maximum likelihood estimator is
asymptotically normal with variance $I(\theta)^{-1}$ (although the
likelihood equation carries extraneous roots with probability not
tending to zero, the number of spurious local maxima being
asymptotically Poisson with mean $1/\pi$~\cite{Reeds1985}).  It is the
moment-based route to geometry, and only that route, which fails here.  With a Gaussian kernel,
all weak moments are finite and the distributional Fisher information
$\wg_J(\mu)$ is positive and smooth---the kernel places the feature
map in a transversal position.

Consider the Cauchy location family
$f(x;\mu) = \pi^{-1}(1+(x-\mu)^2)^{-1}$ with a Gaussian kernel
$\varphi_s(x) = (2\pi s^2)^{-1/2}\exp(-x^2/(2s^2))$.  The weak
moments of all orders are finite:
\[
  \wm{j}(\mu) = \int_{-\infty}^{\infty}
  x^j\, f(x;\mu)\, \varphi_s(x)\, dx < \infty
  \qquad \text{for all } j\geq 0,
\]
since the kernel decay dominates the polynomial tail of the Cauchy.
The distributional Fisher information based on the first $J+1$ weak
moments is
\[
  \wG^{(J)}(\mu) = \sum_{j=0}^{J}
  \left(\frac{\partial \wm{j}}{\partial \mu}\right)^{\!2},
\]
which is strictly positive for all $\mu$ whenever $J\geq 1$.  For the
location family $f(x;\mu) = f_0(x-\mu)$, the derivative is
\[
  \frac{\partial \wm{j}}{\partial \mu}
  \;=\;
  \int_{-\infty}^{\infty}
  x^j\, \frac{\partial f}{\partial \mu}(x;\mu)\,
  \varphi_s(x)\, dx
  \;=\;
  \int_{-\infty}^{\infty}
  x^j\, f(x;\mu)\,\varphi_s(x)\,
  \frac{2(x-\mu)}{1+(x-\mu)^2}\, dx,
\]
where $2(x-\mu)/(1+(x-\mu)^2)$ is the Cauchy score function.  This
integral is well-defined for all $j \geq 0$.  For suitable moment
orders---for instance $j = 1$ under the Gaussian kernel---the
monomial weight $x^j$ breaks the antisymmetry of the score factor
about~$\mu$, producing a non-zero derivative.
This is a concrete instance of the transversality mechanism: the kernel
lifts the feature map out of the rank-degeneracy stratum (Type~II),
restoring a non-degenerate metric on the parameter space.

\subsection{Elliptically contoured distributions: representation
  degeneracy}
\label{sec:elliptical}

\begin{example}[Elliptically contoured distribution]
\label{ex:elliptical}
Let $X\in\R^d$ have characteristic function
$\varphi_X(u) = e^{iu^\top\mu}\psi(u^\top\Sigma u)$, where the
radial profile~$\psi$ does not admit closed-form Fourier inversion.
The model $\{T_\theta : \theta = (\mu,\Sigma)\} \subset \cS'(\R^d)$
is a distributional family without classical densities---a Type~0
(representation) degeneracy.

For a kernel $\varphi\in\cS(\R^d)$, the feature map based on first
and second weak moments is
\[
  \Phi_\varphi(\theta)
  = \bigl(\langle T_\theta, x_i\varphi\rangle,\;
    \langle T_\theta, x_i x_k\varphi\rangle\bigr)_{i,k=1}^{d}.
\]
This map is well-defined and smooth in $\theta$ by the continuity of
tempered distributions on Schwartz space
(Remark~\ref{rem:continuity_perturbation}), even though no density is
available.  For a positive kernel, the first and second weak moments
recover the location~$\mu$ and scale~$\Sigma$ up to identifiable
transformations.  The kernel thus \emph{creates} a smooth embedding
where none existed classically, resolving the representation
degeneracy.

When the kernel belongs to a parametric family
$\{\varphi_\lambda\}_{\lambda\in\Lambda}$, the joint map
$F(\theta,\lambda) = \Phi_{\varphi_\lambda}(\theta)$ satisfies the
hypotheses of Theorem~\ref{thm:main}, and the transversality
conclusion holds for generic~$\lambda$.
\end{example}

\subsection{Non-dominated families: three mechanisms and a Gaussian-measure
  example}
\label{sec:nondominated}

The representation degeneracy of the elliptical example is the mildest case of
a broader phenomenon: families for which \emph{no single dominating measure
exists}, so that the classical likelihood cannot be formed at all.  These are a
Type~0 degeneracy par excellence---the classical feature map is undefined---and
they arise through three mechanisms.

\begin{enumerate}
\item[(M1)] \emph{Parameter-dependent support.}  The support of $P_\theta$
  moves with~$\theta$, as for $\delta_\theta$ or the mixed law
  $\tfrac12\delta_\theta+\tfrac12 N(0,1)$; a common dominating measure would
  need an atom at every $\theta\in\R$.  This case is finite-dimensional: the
  instrument (point evaluation against $x^j\varphi$) turns the moving atom into
  a smooth weak feature, and the analysis is rigorous (the moving-atom
  estimator of~\cite{B}).  Atoms are not required for the phenomenon:
  the location family generated by the Cantor distribution is singular
  continuous, yet the set $\{\theta : P_\theta \ll \nu\}$ is
  Lebesgue-null for \emph{every} $\sigma$-finite~$\nu$, so the family
  is undominated while bounded instruments read closed-form weak
  metrics from it~\cite{E}.
\item[(M2)] \emph{Infinite-dimensional singularity.}  Distinct parameters
  induce mutually singular laws on a function space---$\mathrm{Unif}(0,\theta)^{\otimes\infty}$,
  Brownian motions with different volatilities, or Gaussian measures with
  different covariances---by Kakutani's dichotomy and the Feldman--H\'ajek
  theorem.  Here the instrument is a finite-dimensional observation operator
  (below).
\item[(M3)] \emph{Model uncertainty.}  The model is a neighbourhood rather than
  a curve---Huber contamination $\{(1-\varepsilon)P_0+\varepsilon Q\}$,
  uncertain-volatility and multiple-prior families---and is non-dominated
  because it contains, for instance, all point masses.  A bounded instrument
  induces a dominated \emph{observed} model even when the original neighbourhood
  is not.
\end{enumerate}

In each case the resolution is the same in spirit: although the family of full
laws admits no common density, the \emph{instrument-mediated} feature
map---weak moments, or the pushforward under an observation operator---is
well-defined, and for a generic instrument it is transversal to the degeneracy
strata.  Non-domination is a property of the infinite-resolution object; every
finite-resolution observation is dominated and, generically, transversal.

\paragraph{Gaussian measures on a function space.}
We make mechanism~(M2) concrete.  Let $H$ be a separable Hilbert space and let
$\gamma_\theta = N(0, C_\theta)$ be centred Gaussian measures with trace-class
covariance operators $C_\theta$, $\theta\in\Theta\subset\R^p$.  By the
Feldman--H\'ajek theorem two such measures are either equivalent or mutually
singular, and they are mutually singular unless the covariances are
commensurate in the precise sense of Feldman--H\'ajek (a Hilbert--Schmidt
condition relating $C_{\theta_1}$ and $C_{\theta_2}$).  For a generic
parametrisation of the covariance, distinct $\theta$ therefore give mutually
singular $\gamma_\theta$, so $\{\gamma_\theta\}$ is non-dominated and no
likelihood exists on~$H$.

An observation operator resolves this at finite resolution.  Fix test
directions $e_1,\dots,e_K\in H$ (for instance the leading Karhunen--Lo\`eve
eigenvectors, or smooth test functions) and let $\mathcal O_K:H\to\R^K$,
$\mathcal O_K(x)=(\langle x,e_1\rangle,\dots,\langle x,e_K\rangle)$.  The
pushforward $(\mathcal O_K)_*\gamma_\theta = N(0,\Sigma_\theta^{(K)})$, with
$(\Sigma_\theta^{(K)})_{kl}=\langle C_\theta e_k,e_l\rangle$, is a
finite-dimensional Gaussian, and any two such pushforwards with non-degenerate
covariances are mutually absolutely continuous: the \emph{observed} model is
dominated.  The feature map $\theta\mapsto\Sigma_\theta^{(K)}$ is smooth, and
the finite-dimensional transversality theorem (Theorem~\ref{thm:main}) applies:
for a generic choice of test directions and $K$ large enough that
$\theta\mapsto(\langle C_\theta e_k,e_l\rangle)_{k\le l\le K}$ separates the
parameter, the map is an immersion on compact subsets of~$\Theta$, so the
covariance parameters are identifiable at finite resolution.

\begin{remark}[The infinite-resolution limit is the hard problem]
\label{rem:infinite_resolution}
As $K\to\infty$ the observed experiments approach the singular, non-dominated
regime, and the finite-dimensional argument no longer applies directly: one
re-enters the Fr\'echet/Banach setting of Section~\ref{sec:infinite_dim}, where
the relevant statements (Sard--Smale, Abraham, Quinn) rest on Fredholm
hypotheses that are non-trivial to verify.  The rigorous content of this
example is therefore entirely finite-dimensional---each $\mathcal O_K$ yields a
dominated, transversal experiment---while a genuine \emph{infinite-dimensional}
transversality theory for non-dominated families remains open
(cf.\ Section~\ref{sec:infinite_dim} and Principle~\ref{conj:generic}).
\end{remark}

\subsection{Non-chordal graphical models: higher-order instability}
\label{sec:graphical}

\begin{example}[Non-chordal graphical model]
\label{ex:graphical}
Consider a four-variable Gaussian graphical model with graph
$X_1\!-\!X_2\!-\!X_3\!-\!X_4\!-\!X_1$ (a chordless $4$-cycle).
The concentration matrix $K = \Sigma^{-1}$ is constrained to have
$K_{13} = K_{24} = 0$, but the model is not decomposable.  In
decomposable Gaussian graphical models the likelihood factorises
according to cliques and separators, yielding explicit formulae for the
maximum likelihood estimator in terms of marginal covariance matrices.
In non-decomposable models, such as the present four-cycle, the
likelihood is still perfectly explicit, but the clique--separator
factorisation is unavailable; consequently the likelihood equations
couple the missing-edge constraints globally, and the maximum likelihood
estimator is typically obtained by iterative convex optimisation rather
than by a closed-form decomposable formula.

This should be read not as a failure of likelihood \emph{existence} but
as a failure of \emph{decomposability}, and two precise statements
calibrate what the obstruction is---and what it is not.  First, no
first-order degeneracy is present: for \emph{every} graph, chordal or
not, the Gaussian graphical model is a regular exponential family
whose canonical parameter comprises the free entries of~$K$, so the
classical moment map is a real-analytic embedding with everywhere
nonsingular Jacobian---degeneracies of Types~I and~II simply do not
occur.  The weak feature map $\Phi_\varphi$ inherits this first-order
regularity for Gaussian kernels of sufficiently large scale, and for
generic kernels by Theorem~\ref{thm:main}.  Second, the documented
pathologies of the four-cycle---the absence of a clique--separator
formula, the slow convergence of iterative proportional scaling, the
instability of regularised estimation near the model boundary---are
quantitative and of higher order: along the missing-edge directions
the constraint geometry couples globally around the cycle, and the
second-order data of the feature map (the second fundamental form of
the immersion, equivalently the curvature and the conditioning of the
induced weak metric) deteriorate.  The model \emph{approaches} a
Type~IV stratum without lying in it: the obstruction is a
near-tangency to a higher-order degeneracy stratum, not membership in
it.

A kernel with rapid decay damps the extreme configurations that drive
this ill-conditioning, increasing the distance of the higher jets
$j^r\Phi_\varphi$ ($r \ge 2$) to the Type~IV strata; and should an
exact higher-order degeneracy occur at particular parameter values,
Theorem~\ref{thm:main} removes it for a generic kernel in a
sufficiently rich family.  Kernel regularisation is thus particularly
natural for non-decomposable graphical models, where the
clique--separator factorisation---and with it any closed-form
decomposable maximum likelihood estimator---is unavailable.
\end{example}

\section{Weak Stein Geometry and Transversality}
\label{sec:stein}

A \emph{Stein operator} for $P_\theta$ is a linear operator
$A_\theta$ such that $\E_\theta[A_\theta g(X)]=0$ characterises
$P_\theta$~\cite{Stein1972,SteinChenGoldstein2004}.  In the
distributional framework, the \emph{weak Stein functional} is
$\mathcal{S}_\varphi(T,\theta)(g) =
\langle T, A_\theta(g\varphi)\rangle$, and the \emph{weak Stein
discrepancy} is
\[
  {}^{(\varphi)}\mathcal{D}(T,T_\theta)
  \;=\;
  \sup_{g\in\mathcal{H}}
  |\langle T, A_\theta(g\varphi)\rangle|.
\]
The model is the zero set $\mathcal{S}_\varphi^{-1}(0)$.
If $\mathcal{S}_\varphi$ is transversal to the zero section, the
model is a smooth submanifold (preimage theorem), and the discrepancy
is a well-behaved diagnostic.  The kernel enters as a parameter, and
the parametric transversality theorem guarantees that a generic
$\varphi$ ensures regularity.

\section{The Behrens--Fisher Problem}
\label{sec:behrens_fisher}

\subsection{The classical problem}

Let $X_1,\ldots,X_m\sim N(\mu_1,\sigma_1^2)$ and
$Y_1,\ldots,Y_n\sim N(\mu_2,\sigma_2^2)$ be independent, with
$\sigma_1^2\neq\sigma_2^2$ both unknown.  Testing $H_0:\mu_1=\mu_2$
is the Behrens--Fisher problem.  No pivotal quantity exists whose
distribution is free of the nuisance ratio
$\rho=\sigma_1^2/\sigma_2^2$~\cite{Fisher1935,Fisher1939,
Jeffreys1961,Welch1947}.

\subsection{Geometric interpretation and regularisation}

The full model has $\theta=(\mu_1,\mu_2,\sigma_1,\sigma_2)$; the null
hypothesis is the submanifold
$\Theta_0 = \{\mu_1=\mu_2\}\cong\R\times(0,\infty)^2$.  The
difficulty is that the projection of the sufficient-statistic manifold
onto the testing direction $\mu_1-\mu_2$ is not transversal to
$\Theta_0$ in the product space of test statistics and nuisance
parameters: the distribution of any test statistic depends on~$\rho$.

\begin{quote}
\emph{The Behrens--Fisher problem is a transversality failure: the
classical feature map does not place the null hypothesis in a generic
position relative to the nuisance parameter structure.}
\end{quote}

The kernel provides a family of deformations (indexed by~$s$) of the
feature map: for generic~$s$, the deformed null hypothesis is
transversal (Theorem~\ref{thm:main}), while in the limit
$s\to\infty$ (classical case) transversality is lost.  The
``paradox'' dissolves once one recognises that the classical framework
corresponds to a degenerate point in the space of representations.

\subsubsection{Explicit regularisation}

With a Gaussian kernel $\varphi_s(x) = (2\pi s^2)^{-1/2}
e^{-x^2/(2s^2)}$, the zeroth weak moments of the two populations are
\begin{equation}\label{eq:bf_w0}
  {}^{(\varphi_s)}m_0^{(k)}
  \;=\;
  \frac{1}{\sqrt{\sigma_k^2 + s^2}}\,
  \exp\!\left(-\frac{\mu_k^2}{2(\sigma_k^2+s^2)}\right),
  \qquad k=1,2.
\end{equation}
The location and scale parameters are \emph{coupled} through
$\sigma_k^2+s^2$.  When $s^2\gg\max(\sigma_1^2,\sigma_2^2)$,
\[
  {}^{(\varphi_s)}m_0^{(k)}
  \;\approx\;
  \frac{1}{s}\,e^{-\mu_k^2/(2s^2)},
\]
and the nuisance parameters effectively disappear.  Under $H_0$
($\mu_1=\mu_2$), the difference
$\Delta\,{}^{(\varphi_s)}m_0 = {}^{(\varphi_s)}m_0^{(1)} - {}^{(\varphi_s)}m_0^{(2)}$
converges to zero uniformly in $(\sigma_1,\sigma_2)$.  In this
regime, the null hypothesis becomes a regular submanifold of the
feature space, transversal to the nuisance-parameter fibration.

There is a trade-off: large $s$ gives nuisance insensitivity but
reduces statistical power (the feature map becomes ``coarse''),
analogous to the efficiency--robustness trade-off of~\cite{B}.

\begin{remark}
Fisher's fiducial argument~\cite{Fisher1935} and Jeffreys' Bayesian
approach~\cite{Jeffreys1961} are alternative regularisation strategies
for the same non-transversality.  The kernel-based approach is
directly connected to transversality theory, which guarantees
effectiveness for generic kernels.
\end{remark}

\begin{remark}
The transversality perspective extends to Behrens--Fisher-type
problems in non-normal settings.  The kernel provides a common
representation in which the null hypothesis can be formulated and
tested uniformly across distributional families.
\end{remark}

\section{Towards a Singularity Classification}
\label{sec:classification}

Let $\Theta$ have dimension~$p$ and use $K+1$ weak moments as
features.  \textbf{Non-identifiability} (Type~I) has codimension
$K+1$ in the multijet space; for $K+1>2p$, transversality implies
injectivity.  \textbf{Rank drop} (Type~II): the Thom--Boardman
stratum $\Sigma^1$ has codimension $K+1-p+1$; for $K+1>2p-1$,
generically $\det\wG>0$ everywhere
(Corollary~\ref{cor:full_rank}).  \textbf{Higher singularities} have
progressively higher codimension.

For a $p$-parameter model, $K+1\ge 2p$ weak moments generically
ensure identifiability, and $K+1\ge 2p+1$ ensure information
regularity.  In practice, the weak characteristic function (an
infinite-dimensional feature) easily satisfies these bounds.

\section{The Unifying Perspective}
\label{sec:unification}

The transversality perspective unifies several threads in statistical
inference.

\subsection{Identifiability as injectivity}

Identifiability $\Leftrightarrow$ injectivity of $\Phi_\varphi$
$\Leftrightarrow$ avoidance of the self-intersection diagonal.
Transversality gives generic injectivity.  This subsumes the
moment-determinacy results of~\cite{Stoyanov2000}, the weak
identifiability theorems of~\cite[Theorems~6.2, 6.4 and~6.8]{A}, and the singular-limit
analysis of~\cite{E}.

\subsection{Robustness as metric boundedness}

The distributional metric $\wg_J(\mu)$ is bounded and the manifold
has finite geodesic length.  This reflects the controlled behaviour
of the feature map under kernel regularisation.

\subsection{Information geometry as Riemannian geometry of the
  feature map}
\label{sec:info_geom}

The distributional metric tensor is the first fundamental form of the
immersion $\Phi_\varphi$:
\[
  \wG_{ab}(\theta)
  = (D\Phi_\varphi)^\top(D\Phi_\varphi).
\]
The classical information geometry of
Amari~\cite{Amari1985} and
Barndorff-Nielsen~\cite{BarndorffNielsen1978} is the special case
$\varphi\equiv 1$; the distributional information geometry
of~\cite{C} is the general case.

\subsection{Regularisation as deformation}

The kernel space $\cS(\R^d)$ is the deformation space.  Optimal
regularisation $=$ the kernel that best approximates the classical
problem while remaining transversal.

\subsection{Stein's method as zero-section transversality}

The weak Stein discrepancy measures distance from the zero set of the
Stein map.  Transversality ensures regularity.

\subsection{The Behrens--Fisher problem as nuisance non-transversality}

The classical Behrens--Fisher ``paradox'' is a non-transversality of
the null hypothesis relative to the nuisance-parameter fibration.
The kernel resolves it by deforming the feature map into a generic
position.

\subsection{Inferential separation as transversality}
\label{sec:inferential_separation}

The classical theory of inferential separation---encompassing
sufficiency, ancillarity, and the nonformation
principle~\cite{BarndorffNielsen1978,JorgensenLabouriau2012}---admits a
natural transversality interpretation within the distributional
framework.

Consider a parametric model with $\theta=(\alpha,\beta)$, where
$\alpha$ is the parameter of interest and $\beta$ is a nuisance
parameter (we follow the interest--nuisance notation of the companion
papers \cite{C,E}).  The
full feature map decomposes as
$\Phi_\varphi(\alpha,\beta) \in \mathcal{F}$, and the nuisance
parameter defines a \emph{fibration}
$\pi:\Theta\to A$, $\pi(\alpha,\beta)=\alpha$, whose fibres are the
nuisance orbits $\{\alpha\}\times B$.

\emph{Inferential separation} requires that inference about $\alpha$ be
insensitive to the nuisance parameter~$\beta$.  In the feature-map
language, this means that the restriction of $\Phi_\varphi$ to the
interest direction is transversal to the nuisance fibres:
\begin{quote}
\emph{Inferential separation holds when the feature map is transversal
to the nuisance-parameter fibration, so that the image of the
interest subspace intersects the nuisance tangent space only
trivially.}
\end{quote}

In the inference-functional formulation of~\cite{C}, an inference
functional $\Psi(Y,\theta)$ for the interest parameter lives in the
orthogonal complement $\mathcal{T}_N^\perp$ of the nuisance tangent
space.  The Bhapkar--Godambe
projection~\cite{Godambe1960,JorgensenLabouriau2012} is the
constructive mechanism that enforces this orthogonality: it projects
an arbitrary quasi-inference functional onto $\mathcal{T}_N^\perp$,
which is precisely the operation of deforming the inference functional
into a transversal position relative to the nuisance directions.

In the distributional setting, sinusoidal inference functionals
$\psi_c(x,\mu) = \sin(c(x-\mu))$ for symmetric location-scale
models satisfy the orthogonality condition \emph{automatically},
without requiring explicit projection (this is proved, together with
a block-diagonality statement for the joint Godambe metric, in
\cite{E}).  This automatic
transversality arises from the symmetry of the characteristic
function: the kernel structure of the sinusoidal inference functional
places it naturally in a position transversal to the scale nuisance
tangent space.  This is a distributional phenomenon that holds even
for heavy-tailed models such as the Cauchy, where moment-based
inference functions are unavailable (the Cauchy score itself exists
and is even bounded), and it extends verbatim to families with no
score at all, such as the singular-continuous Cantor location--scale
model of~\cite{E}.

The hierarchy of nonformation concepts---S-nonformation,
I-nonformation, and L-nonforma\-tion corresponds to progressively
weaker transversality conditions on the feature map relative to the
nuisance fibration:
\begin{itemize}
\item \textbf{S-nonformation:} the feature map factors completely
  through the interest projection (full transversality to nuisance
  fibres).
\item \textbf{I-nonformation:} the conditional feature map, given the
  nuisance component, is saturated (local transversality at the
  maximum).
\item \textbf{L-nonformation:} the profile of the feature map
  depends on the data only through a reduction (transversality at
  the level of the normed profile).
\end{itemize}
This geometric reading shows that the various notions of data
reduction and inferential separation are manifestations of a single
principle: the feature map being in generic position relative to the
nuisance structure.  The kernel provides a mechanism for achieving
this transversality even in models where classical likelihood-based
separation fails.

\section{A Microlocal Perspective}
\label{sec:microlocal}

The geometric picture developed above admits a refinement through the
\emph{microlocal} analysis of distributions.  The structure theorem
(Remark~\ref{rem:structure_theorem}) explains that a tempered distribution is
built from differentiated regularity, but it does not record \emph{where} the
singularities sit or \emph{in which directions} they propagate; microlocal
analysis supplies exactly this information.

For $T\in\cS'(\R^d)$ the singular support $\operatorname{sing\,supp}(T)$
records the points near which $T$ fails to be smooth, and the wave front set
$\mathrm{WF}(T)\subset \R^d\times(\R^d\setminus\{0\})$ refines it by recording
also the cotangent directions in which the singularity persists (the wave
front set is due to H\"ormander; see Strichartz~\cite{Strichartz2003}).  In
statistical terms this distinguishes not merely \emph{which} regions of sample
space carry singular behaviour, but \emph{which directional features} of the
model carry the singular information.

From the present viewpoint the kernel---the observational instrument---is a
device that attenuates or resolves selected microlocal features of the law:
pairing with $g\varphi$ tests $T$ against a smooth, rapidly decaying weight, so
the feature map $\Phi_\varphi$ reads off a finite-dimensional projection of the
microlocal structure.  Different kernels (different instruments) probe
different components of the wave front set, just as different physical
instruments are sensitive to different resolutions or directions of variation.

This suggests a microlocal reading of the degeneracy stratification of
Section~\ref{sec:degeneracies}.  Non-identifiability, moment indeterminacy, and
higher-order singular dependence (Types~I, III and~IV) may be viewed not only
as degeneracies of a finite-dimensional feature map, but as failures to
\emph{separate} microlocal features of the underlying law.  Transversality then
plays the role of a generic \emph{separation principle}: after a generic
perturbation of the instrument, the feature map should meet the relevant
microlocal degeneracy strata in a stable, non-tangential way.  Making this
correspondence rigorous---a microlocal transversality theory for distributional
models---requires the infinite-dimensional machinery discussed in
Section~\ref{sec:infinite_dim} and lies beyond the present scope.  We record it
as the natural geometric home for the wave front structure of the weak
framework, and as a direction for future work.

\section{Discussion: The Distributional Programme and Transversality}
\label{sec:discussion-programme}

The transversality perspective developed in this paper provides a unifying
geometric interpretation of a broader programme on distributional statistical
models. This programme is developed across a series of companion papers, each
addressing a different aspect of the framework.

\medskip

Paper~A~\cite{A} introduces the foundational construction of
distribution--kernel pairs $(T,\varphi)$, weak moments, and weak
characteristic functions, and establishes uniqueness results
(\cite[Theorems~6.2, 6.4 and~6.8]{A}) showing that positive kernels resolve
classical moment indeterminacy under mild conditions. Paper~B~\cite{B} develops the associated statistical methodology,
including estimation via weak moment matching, robustness properties arising
from kernel decay, asymptotic theory, and density reconstruction.

\medskip

Paper~C~\cite{C} extends the classical theory of inference functions
(Godambe, 1960) to distributional models by introducing \emph{observation
operators} $\mathcal{O} : \cS'(\R) \to \mathcal{Y}$ that map distributional
models to an observation space, and defining \emph{inference functionals}
$\Psi(Y, \theta)$ on the observed data $Y = \mathcal{O}(T_X)$.  The
framework organises inference into three conceptual layers: at the most
general level (Layer~I), observations are the output of linear operators
acting on~$\cS'$; at an intermediate level (Layer~II), observations are
classical point values but the inference constructions are distributional; at
the most classical level (Layer~III), the score equation is recovered as a
special case.
This inference-functional framework can be interpreted as the statistical
realisation of the geometric picture developed here: inference functionals
define coordinates of the kernel-induced feature map, and their asymptotic
properties follow from its regularity.  Classical pathologies---heavy tails,
lack of densities, and nuisance effects---appear as failures of
transversality, while the construction and transformation of inference
functionals correspond to selecting coordinates that are transversal to
degeneracy and nuisance directions.  In particular, orthogonalisation
procedures such as Bhapkar--Godambe projections enforce transversality to
nuisance tangent spaces, thereby ensuring stability and efficiency of
inference.  Furthermore, the theory of inferential
separation---encompassing the nonformation principle and its hierarchy of
sufficiency concepts (S-, I-, and
L-nonformation)~\cite{JorgensenLabouriau2012}---receives a geometric
interpretation as transversality of the feature map relative to the
nuisance-parameter fibration (Section~\ref{sec:inferential_separation}).
Sinusoidal inference functionals in symmetric models achieve this
transversality automatically, without explicit projection.

\medskip

Paper~D~\cite{D} extends the framework to goodness-of-fit and minimum
discrepancy estimation via weak Stein discrepancies. From the present
viewpoint, Stein identities define constraint maps whose regularity is again
governed by transversality, linking discrepancy-based inference to the same
geometric principles.

\medskip

Paper~E~\cite{E} develops the differential-geometric structure of the
framework: any model represented by tempered distributions and equipped
with an instrument of nonsingular sensitivity carries a
Godambe--Riemannian metric, dominated by the Fisher metric in the
Loewner order whenever the latter exists, with worked examples beyond
the Fisher--Rao class (including an undominated Cantor location family
and an $\alpha$-stable lattice model) and a geometric characterisation
of inferential separation as block-diagonality of the metric. In the
transversality interpretation, this geometry arises as the pullback of
a regular metric on the feature space via a transversal feature map,
connecting the theory to classical information geometry
\cite{Amari1985,BarndorffNielsen1978}; the regularity hypotheses of the
metric theorems in~\cite{E} are exactly transversality conditions on
the feature map, verified there by direct computation and here shown
to hold generically.

\medskip

Taken together, these works develop a programme in which statistical models
are represented through distribution--kernel pairs and analysed via the
geometry of the induced feature maps.  The present paper provides the
overarching geometric principle: kernels act as generic perturbations that
enforce transversality, thereby explaining identifiability, robustness,
regularity, and stability within a single unified framework.
Further case studies---including the log-normal family, graphical
models, and the singular limit to classical moments---are in
preparation.

The core results of this paper are developed for parametric models with
finite-dimensional parameter spaces, where the classical transversality
theorems of Thom and Mather apply directly.  For the infinite-dimensional
case---semiparametric and nonparametric models, and genuinely non-dominated
families---we do not stop at a deferral: Theorem~\ref{thm:infinite} gives a
\emph{conditional} transversality theorem in the Banach-manifold framework of
Smale~\cite{Smale1965} and Abraham~\cite{Abraham1963}, under an explicit
Fredholm hypothesis on the feature map, and Example~\ref{ex:gaussian-sequence}
verifies that hypothesis in closed form for a non-dominated Gaussian sequence
model with an infinite-dimensional parameter.  The entire additional difficulty
is thereby isolated in a single analytic input---the Fredholm property---whose
verification across broad model classes, together with a fully topology-free
treatment, remains the principal open problem.  The guiding picture---the kernel
as a generic perturbation enforcing transversality---thus carries over to the
infinite-dimensional setting wherever that property can be established.  A
concluding microlocal perspective (Section~\ref{sec:microlocal}) indicates how
the wave front set refines the degeneracy stratification, and marks the natural
geometric horizon of the programme.

A distinctive feature of this paper is the development of verifiable
transversality conditions
(Section~\ref{sec:transversality_condition_dim}), which bridge the
gap between the abstract genericity results and concrete model
classes.  The component-wise criterion
(Lemma~\ref{lem:component_transversality}) shows that transversality
can be checked by examining the model and kernel contributions to the
Jacobian separately: the model derivatives encode the intrinsic
parametric geometry, while the kernel derivatives provide
supplementary directions that generically restore full rank.  This
mechanism is verified explicitly for location families, the
log-normal, Stein discrepancies, and graphical models, establishing
that the transversality hypothesis of
Theorem~\ref{thm:main} is not merely a mathematical convenience
but holds in a wide range of statistically relevant settings.

\bigskip
\noindent
\textbf{Acknowledgements.}
\emph{I am grateful to J{\o}rgen Hoffmann-J{\o}rgensen, who remarked that a
distributional representation would free classical inference theory from the
arbitrary choice of a representative of the Radon--Nikodym derivative, and to
Ole E.\ Barndorff-Nielsen, who asked whether weak cumulant formal expansions
could be constructed in this setting---both discussions over a good cup of
espresso coffee in my office in Aarhus.  Jordan Stoyanov suggested the key
example of the log-normal family of distributions and called my attention to
the modern theory of M-indeterminacy.  I also thank Isabel Labouriau for
several informal discussions and for following the development of the
series of papers on distributional inference.}


\appendix

\section{Proofs for the log-normal model}
\label{app:lognormal_proofs}

This appendix provides proofs for
Proposition~\ref{prop:lognormal-transversality} (moment degeneracy
breaking, Section~\ref{sec:examples}) and
Proposition~\ref{prop:lognormal_transversality} (transversality
verification, Section~\ref{sec:transversality_condition_dim}).  The
central tool is M-determinacy of the tilted measure.

\subsection*{The tilted measure and its M-determinacy}

Let $P_{\mu,\sigma}$ denote the log-normal distribution with
parameters $(\mu,\sigma)$ and let
\[
\varphi_s(x) = (2\pi s^2)^{-1/2}\exp\bigl(-x^2/(2s^2)\bigr)
\]
be the Gaussian kernel with scale $s>0$.  Define the \emph{tilted measure}
\begin{equation}\label{eq:tilted}
  dQ_{\mu,\sigma,s} \;=\;
  \frac{\varphi_s\, dP_{\mu,\sigma}}{Z(\mu,\sigma,s)},
  \qquad
  Z(\mu,\sigma,s) \;=\;
  \int_0^\infty \varphi_s(x)\, f_{\mu,\sigma}(x)\, dx
  \;=\; \wm{0}(\theta).
\end{equation}
Since $\varphi_s(x) \leq C\,e^{-x^2/(2s^2)}$, the tilted measure
$Q_{\mu,\sigma,s}$ has sub-Gaussian tails.  In particular, all moments
of~$Q$ are finite, and Carleman's condition
\[
  \sum_{j=1}^{\infty}
  \bigl(\E_Q[X^{2j}]\bigr)^{-1/(2j)} = +\infty
\]
is satisfied (the super-polynomial decay of $\varphi_s$ forces the
moments of~$Q$ to grow at most as fast as those of a Gaussian).  Hence
$Q_{\mu,\sigma,s}$ is \emph{M-determinate}: it is the unique
probability measure with its moment sequence.

\subsection*{Proof of
  Proposition~\ref{prop:lognormal-transversality}}

\emph{(i)} The classical moment indeterminacy of the log-normal is
well known~\cite{Stoyanov2000}: the Stieltjes perturbation
$h(x) = \sin(2\pi\ln x)$ satisfies $\int_0^\infty x^j h(x)
f_{\mu,\sigma}(x)\,dx = 0$ for all~$j$, so distinct distributions
can share the same moment sequence.

\emph{(ii)} \emph{(Smoothness and real analyticity.)}
Fix a compact set $K \subset \Theta \times (0,\infty)$.  For
$(\mu,\sigma,s) \in K$, the integrand
$x^j f_{\mu,\sigma}(x)\,\varphi_s(x)$ is dominated by
$C_K\, x^j \exp(-x^2/(2s_{\max}^2))$, which is integrable.  By
differentiation under the integral sign, the map
$(\mu,\sigma,s) \mapsto {}^{(\varphi_s)}\!m_j(\mu,\sigma)$ is
$C^\infty$.  For real analyticity, observe that the log-normal density
$f_{\mu,\sigma}(x) = (x\sigma\sqrt{2\pi})^{-1}\exp(-(\ln x -
\mu)^2/(2\sigma^2))$ extends holomorphically in $(\mu,\sigma)$ to a
complex neighbourhood of any compact subset of~$\Theta$.  On such a
neighbourhood, the modulus of the integrand is dominated uniformly by
the same Gaussian bound, so by holomorphic dominated convergence the
weak moments extend to holomorphic functions of $(\mu,\sigma)$.
Hence $(\mu,\sigma,s) \mapsto {}^{(\varphi_s)}\!m_j$ is real-analytic
on $\Theta \times (0,\infty)$.

\emph{(iii)} \emph{(Injectivity for all $s > 0$.)}
Suppose that ${}^{(\varphi_s)}\!m_j(\theta_1) =
{}^{(\varphi_s)}\!m_j(\theta_2)$ for all $j \geq 0$.  Then in
particular $Z(\theta_1,s) = Z(\theta_2,s)$, so the tilted measures
$Q_{\theta_1,s}$ and $Q_{\theta_2,s}$ have identical moments of all
orders.  By M-determinacy of the tilted measure,
$Q_{\theta_1,s} = Q_{\theta_2,s}$.  Since $\varphi_s > 0$ on
$(0,\infty)$, this implies $P_{\theta_1} = P_{\theta_2}$, and hence
$\theta_1 = \theta_2$ (since the log-normal family is identifiable
in its parameters).

Note that this argument yields injectivity for \emph{every} $s > 0$,
not merely for a generic set of scales.

\emph{(iv)} \emph{(Immersion on compact subsets.)}
Fix $s > 0$ and a compact $\Theta_0 \subset \Theta$.  By
Proposition~\ref{prop:lognormal_transversality} below, for each
$\theta \in \Theta_0$ there exist moment orders $j_0(\theta) <
j_1(\theta)$ such that
$\operatorname{rank} D_\theta \Phi_s(\theta) = 2$, where
$\Phi_s = ({}^{(\varphi_s)}\!m_{j_0}, {}^{(\varphi_s)}\!m_{j_1})$.
The full-rank condition is open, so each $\theta$ has a neighbourhood
$U_\theta$ on which the same pair $(j_0(\theta), j_1(\theta))$
gives rank~$2$.  By compactness of $\Theta_0$, finitely many such
neighbourhoods cover $\Theta_0$; the union of the corresponding
moment orders gives a finite set $\{j_0, \ldots, j_K\}$ such that
$\Phi_s = ({}^{(\varphi_s)}\!m_{j_0}, \ldots,
{}^{(\varphi_s)}\!m_{j_K})$ is an immersion on $\Theta_0$.
\qed

\subsection*{Proof of
  Proposition~\ref{prop:lognormal_transversality}}

\emph{Pointwise rank.}
Fix $s > 0$ and $(\mu,\sigma) \in \Theta$.  We show that there exist
moment orders $j_0 < j_1$ such that the $2\times 2$ Jacobian
\[
  J \;=\;
  \begin{pmatrix}
    \partial_\mu\, {}^{(\varphi_s)}\!m_{j_0}
    & \partial_\mu\, {}^{(\varphi_s)}\!m_{j_1} \\[4pt]
    \partial_\sigma\, {}^{(\varphi_s)}\!m_{j_0}
    & \partial_\sigma\, {}^{(\varphi_s)}\!m_{j_1}
  \end{pmatrix}
\]
has rank~$2$ at $(\mu,\sigma)$.

The log-normal score functions are
$s_\mu(x) = (\ln x - \mu)/\sigma^2$ and
$s_\sigma(x) = (\ln x - \mu)^2/\sigma^3 - 1/\sigma$.
The weak moment derivatives can be written as
\[
  \frac{\partial}{\partial\theta_a}\,
  {}^{(\varphi_s)}\!m_j
  \;=\;
  Z \cdot \E_Q\!\bigl[X^j\, s_a(X)\bigr],
  \qquad a \in \{\mu,\sigma\},
\]
where $Z = \wm{0}$ and $\E_Q$ denotes expectation under the tilted
measure~$Q_{\mu,\sigma,s}$.  Thus
$\det J = Z^2 \det M$, where
\[
  M_{ak} \;=\; \E_Q\!\bigl[X^{j_k}\, s_a(X)\bigr],
  \qquad a \in \{\mu,\sigma\},\; k\in\{0,1\}.
\]
It suffices to show that $\det M \neq 0$ for some pair $(j_0,j_1)$.

Since $Q$ satisfies Carleman's condition, the monomials
$\{x^j : j = 0, 1, 2, \ldots\}$ are total in $L^2(Q)$; that is,
their linear span is dense (see Akhiezer~\cite{Akhiezer1965},
Theorem~2.3.3, or Shohat and
Tamarkin~\cite{ShohatTamarkin1943}, Ch.~II).  The score functions
$s_\mu$ and $s_\sigma$ are linearly independent elements of $L^2(Q)$
(they are polynomials of different degree in $\ln x$, and $Q$~charges
all of $(0,\infty)$).

Suppose for contradiction that $\det M = 0$ for every pair
$(j_0,j_1)$.  Then for each~$j$ with
$\E_Q[X^j\, s_\sigma] \neq 0$, the ratio
$\E_Q[X^j\, s_\mu] / \E_Q[X^j\, s_\sigma]$ equals a constant~$c$
independent of~$j$.  This gives
$\E_Q[X^j\,(s_\mu - c\, s_\sigma)] = 0$ for all~$j$.  By totality
of the monomials in $L^2(Q)$, it follows that
$s_\mu = c\, s_\sigma$ holds $Q$-almost surely, contradicting
the linear independence of the scores.

\emph{Global immersion on compact sets.}
The above argument provides, for each $(\mu,\sigma) \in \Theta$, a
pair $(j_0,j_1)$ such that $\det J(\mu,\sigma) \neq 0$.  Since
$\det J$ depends continuously on $(\mu,\sigma)$, non-vanishing
persists in a neighbourhood of each point.  On any compact
$\Theta_0 \subset \Theta$, a finite subcover gives a finite set of
moment orders $j_0 < \cdots < j_K$ such that the feature map
$\Phi_{\varphi_s} : \Theta_0 \to \R^{K+1}$ has
$\operatorname{rank} D\Phi_{\varphi_s} = 2$ at every point of
$\Theta_0$, i.e.\ $\Phi_{\varphi_s}$ is an immersion on~$\Theta_0$.
\qed



\begin{thebibliography}{99}

\bibitem{Akhiezer1965}
N.\,I.~Akhiezer,
\emph{The Classical Moment Problem and Some Related Questions in
  Analysis},
Oliver \& Boyd, Edinburgh, 1965.

\bibitem{Abraham1963}
R.~Abraham,
Transversality in manifolds of mappings,
\emph{Bull.\ Amer.\ Math.\ Soc.}\ \textbf{69} (1963), 470--474.

\bibitem{Amari1985}
S.-i.~Amari,
\emph{Differential-Geometrical Methods in Statistics},
Lecture Notes in Statistics \textbf{28}, Springer, 1985.

\bibitem{BarndorffNielsen1978}
O.\,E.~Barndorff-Nielsen,
\emph{Information and Exponential Families in Statistical Theory},
Wiley, 1978.

\bibitem{Boardman1967}
J.\,M.~Boardman,
Singularities of differentiable maps,
\emph{Publ.\ Math.\ Inst.\ Hautes \'Etudes Sci.}\ \textbf{33} (1967),
21--57.

\bibitem{Ehresmann1951}
C.~Ehresmann,
Les prolongements d'une vari\'et\'e diff\'erentiable: calcul des
jets, prolongement principal,
\emph{C.\ R.\ Acad.\ Sci.\ Paris} \textbf{233} (1951), 598--600.

\bibitem{Fisher1935}
R.\,A.~Fisher,
The fiducial argument in statistical inference,
\emph{Ann.\ Eugenics} \textbf{6} (1935), 391--398.

\bibitem{Fisher1939}
R.\,A.~Fisher,
The comparison of samples with possibly unequal variances,
\emph{Ann.\ Eugenics} \textbf{9} (1939), 174--180.

\bibitem{Godambe1960}
V.\,P.~Godambe,
An optimum property of regular maximum likelihood estimation,
\emph{Ann.\ Math.\ Statist.}\ \textbf{31} (1960), 1208--1211.

\bibitem{GolubitskyGuillemin1973}
M.~Golubitsky and V.~Guillemin,
\emph{Stable Mappings and Their Singularities},
Graduate Texts in Mathematics \textbf{14}, Springer, 1973.

\bibitem{GuilleminPollack1974}
V.~Guillemin and A.~Pollack,
\emph{Differential Topology},
Prentice-Hall, Englewood Cliffs, NJ, 1974.

\bibitem{Hirsch1976}
M.\,W.~Hirsch,
\emph{Differential Topology},
Graduate Texts in Mathematics \textbf{33}, Springer, 1976.

\bibitem{Jeffreys1961}
H.~Jeffreys,
\emph{Theory of Probability}, 3rd ed., Oxford Univ.\ Press, 1961.

\bibitem{JorgensenLabouriau2012}
B.~J{\o}rgensen and R.~Labouriau,
\emph{Exponential Families and Theoretical Inference},
Monografias de Matem\'atica \textbf{52},
Instituto de Matem\'atica Pura e Aplicada (IMPA), Rio de Janeiro,
2012.

\bibitem{A}
R.~Labouriau (2026A).
\emph{Distributional Statistical Models: Weak Moments, Cumulants, and a Central Limit Theorem},
arXiv:2604.20634 [math.PR]

\bibitem{B}
R.~Labouriau (2026B)
\emph{Weak Moment Methods for Statistical Inference: with an Application to Robust Estimation},
arXiv:2604.23619 [stat.ME]

\bibitem{C}
R.~Labouriau (2026C).
\emph{Inference  Functionals and Observation Operators for Distributional}
Statistical Models.
arXiv:2605.19189 [math.ST].

\bibitem{D}
R.~Labouriau (2026D).
\emph{Weak Stein Discrepancies: Kernel-Regularised Goodness-of-Fit and
Minimum Discrepancy Estimation for Heavy-Tailed Models},
in preparation, 2026.

\bibitem{E}
R.~Labouriau (2026E).
\emph{Weak Information Geometry:
Riemannian Structures from Distributional Inference Functions and Stein Discrepancies},
in preparation, 2026.


\bibitem{Mather1970}
J.\,N.~Mather,
Stability of $C^\infty$ mappings, V: Transversality,
\emph{Advances in Math.}\ \textbf{4} (1970), 301--336.

\bibitem{Mather1971}
J.\,N.~Mather,
Stability of $C^\infty$ mappings, VI: The nice dimensions,
in: \emph{Proceedings of the Liverpool Singularities Symposium~I},
Lecture Notes in Math.\ \textbf{192}, Springer, 1971, pp.~207--253.

\bibitem{Quinn1979}
F.~Quinn,
Transversal approximation on Banach manifolds,
in: \emph{Global Analysis}, Amer.\ Math.\ Soc., 1979, pp.~213--222.

\bibitem{Reeds1985}
J.\,A.~Reeds,
Asymptotic number of roots of Cauchy location likelihood equations,
\emph{Ann.\ Statist.}\ \textbf{13} (1985), no.~2, 775--784.

\bibitem{Smale1965}
S.~Smale,
An infinite dimensional version of Sard's theorem,
\emph{Amer.\ J.\ Math.}\ \textbf{87} (1965), 861--866.

\bibitem{Stein1972}
C.~Stein,
A bound for the error in the normal approximation to the distribution
of a sum of dependent random variables,
in: \emph{Proc.\ Sixth Berkeley Symp.}, Vol.~II, 1972, pp.~583--602.

\bibitem{SteinChenGoldstein2004}
C.~Stein, L.\,H.\,Y.~Chen, and L.~Goldstein,
Normal approximation,
in: \emph{An Introduction to Stein's Method},
Singapore Univ.\ Press, 2005, pp.~1--59.

\bibitem{ShohatTamarkin1943}
J.\,A.~Shohat and J.\,D.~Tamarkin,
\emph{The Problem of Moments},
Mathematical Surveys \textbf{1},
Amer.\ Math.\ Soc., 1943.

\bibitem{Stoyanov2000}
J.\,M.~Stoyanov,
Krein condition in probabilistic moment problems,
\emph{Bernoulli} \textbf{6} (2000), 939--949.

\bibitem{Strichartz2003}
R.\,S.~Strichartz,
\emph{A Guide to Distribution Theory and Fourier Transforms},
World Scientific, 2003.

\bibitem{Thom1954}
R.~Thom,
Quelques propri\'et\'es globales des vari\'et\'es diff\'erentiables,
\emph{Comment.\ Math.\ Helv.}\ \textbf{28} (1954), 17--86.

\bibitem{Welch1947}
B.\,L.~Welch,
The generalization of `Student's' problem when several different
population variances are involved,
\emph{Biometrika} \textbf{34} (1947), 28--35.

\bibitem{Whitney1944}
H.~Whitney,
The singularities of a smooth $n$-manifold in $(2n-1)$-space,
\emph{Ann.\ of Math.}\ \textbf{45} (1944), 247--293.

\bibitem{Whitney1965}
H.~Whitney,
Tangents to an analytic variety,
\emph{Ann.\ of Math.}\ \textbf{81} (1965), 496--549.

\end{thebibliography}
\end{document}